\documentclass[a4paper, 12pt]{amsart}

\usepackage{cite}

\pdfoutput=1
\usepackage{latexsym, amsmath, amssymb, amsthm, mathrsfs, bm, mathtools,comment}

\usepackage[mathscr]{eucal}
\usepackage{subfigure}

\usepackage{array}
\usepackage{tikz}
\usetikzlibrary{cd}
\usepackage{aliascnt}
\usepackage{graphicx}
\usepackage[T1]{fontenc}
\usepackage{mathptmx}
\usepackage{microtype}
\usepackage{enumerate, enumitem}
\usepackage[colorlinks=true,linkcolor=blue,urlcolor=blue]{hyperref}

\usepackage[onehalfspacing]{setspace}

\usepackage[inner=2.4cm,outer=2.4cm, bottom=3.2cm]{geometry}

\DeclareMathAlphabet{\altmathcal}{OMS}{cmsy}{m}{n}

\allowdisplaybreaks

\usepackage{xcolor, color, soul}

\usepackage{color}
\usepackage[all]{xy}
\usepackage{enumerate}
\usepackage{mathbbol,mathtools}

\newtheorem{theorem}{Theorem}[section]

\newtheorem{headthm}{Theorem}

\newaliascnt{headcor}{headthm}
\newtheorem{headcor}[headcor]{Corollary}
\aliascntresetthe{headcor}

\newaliascnt{headconj}{headthm}

\aliascntresetthe{headconj}

\newaliascnt{corollary}{theorem}
\newtheorem{corollary}[corollary]{Corollary}
\aliascntresetthe{corollary}

\newaliascnt{claim}{theorem}

\aliascntresetthe{claim}

\newaliascnt{lemma}{theorem}
\newtheorem{lemma}[lemma]{Lemma}
\aliascntresetthe{lemma}

\newaliascnt{conjecture}{theorem}

\aliascntresetthe{conjecture}

\newaliascnt{proposition}{theorem}
\newtheorem{proposition}[proposition]{Proposition}
\aliascntresetthe{proposition}

\theoremstyle{definition}
\newaliascnt{definition}{theorem}

\aliascntresetthe{definition}

\newaliascnt{notation}{theorem}

\aliascntresetthe{notation}

\newaliascnt{example}{theorem}
\newtheorem{example}[example]{Example}
\aliascntresetthe{example}

\newaliascnt{examples}{theorem}

\aliascntresetthe{examples}

\newaliascnt{remark}{theorem}
\newtheorem{remark}[remark]{Remark}
\aliascntresetthe{remark}

\newaliascnt{fact}{theorem}

\aliascntresetthe{fact}

\newaliascnt{question}{theorem}
\newtheorem{question}[question]{Question}
\aliascntresetthe{question}

\newaliascnt{questions}{theorem}

\aliascntresetthe{questions}

\newaliascnt{problem}{theorem}

\aliascntresetthe{problem}

\newaliascnt{construction}{theorem}

\aliascntresetthe{construction}

\newaliascnt{setup}{theorem}

\aliascntresetthe{setup}

\newaliascnt{algorithm}{theorem}

\aliascntresetthe{algorithm}

\newaliascnt{observation}{theorem}

\aliascntresetthe{observation}

\newaliascnt{discussion}{theorem}

\aliascntresetthe{discussion}

\newaliascnt{defprop}{theorem}

\aliascntresetthe{defprop}

\def\equationautorefname~#1\null{(#1)\null}
\def\sectionautorefname~#1\null{Section #1\null}
\def\subsectionautorefname~#1\null{\S #1\null}

\def\trdeg{{\rm trdeg}}

\definecolor{myorange}{RGB}{255, 160, 70}

\usepackage[normalem]{ulem}
\newcommand{\stkout}[1]{\ifmmode\text{\sout{\ensuremath{#1}}}\else\sout{#1}\fi}

\def \height{{\operatorname{ht}}}

\def \Spec{{\operatorname{Spec}}}

\def \Proj{{\operatorname{Proj\, }}}
\def\Supp{\operatorname{Supp\, }}

\def \trdeg{{\operatorname{trdeg}}}

\def \divv{{\operatorname{div}}}

\DeclareMathOperator{\Quot}{Quot}

\DeclareMathOperator{\Image}{Image}
\DeclareMathOperator{\Div}{Div}

\DeclareMathOperator{\Bs}{Bs}
\DeclareMathOperator{\sBL}{sBL}
\DeclareMathOperator{\BL}{BL}

\def \f1{\mathbf{1}}

\newcommand{\kk}{\mathbb{k}}

\def\xi{x}

\def\ls{\leqslant}
\def\gs{\geqslant}

\def\fm{\mathfrak{m}}

\def\fm{\mathfrak{m}}

\def\m{\mathfrak{m}}

\def \PP{\mathbb P}
\def \CC{\mathbb C}
\def \QQ{\mathbb Q}
\def \RR{\mathbb R}

\def \NN{ {\mathbb Z}_{\gs 0}}

\def \ZZ{\mathbb Z}

\def \II{\mathbb I}

\def \O{\mathcal O}

\def \M{\mathcal M}

\def \L{\mathcal L}
\def \O{\mathcal O}

\begin{document} 


\title[Degree functions of  graded families of ideals]{Degree functions of  graded families of ideals}

\author[Steven Dale Cutkosky]{Steven Dale Cutkosky$^1$}
\address{Steven Dale Cutkosky\\
	Department of Mathematics, University of Missouri, Columbia,
	MO 65211, USA. \emph{Email:} {\rm cutkoskys@missouri.edu}}
\thanks{$^{1}$ The first author was partially funded by NSF Grant DMS \#2348849.}

\author[Jonathan Monta{\~n}o]{Jonathan Monta{\~n}o$^2$}
\address{Jonathan Monta{\~n}o\\School of Mathematical and Statistical Sciences, Arizona State University, P.O. Box 871804, Tempe, AZ 85287-18041, USA. \emph{Email:} {\rm montano@asu.edu}}
\thanks{$^{2}$ The second author was partially funded by NSF Grant DMS \#2401522.}

\begin{abstract}

We express multiplicities and degree functions of graded families of $\fm_R$-primary ideals in an excellent normal local ring  $(R,\fm_R)$ as  limits of intersection products. Moreover, in dimension 2, we show more refined results  for  divisorial filtrations. Finally, also in dimension 2, we give an example of  a non-Noetherian divisorial filtration $\{I_n\}_{n\gs 0}$  of $\fm_R$-primary ideals  such that the union of all the sets of Rees valuations of all the $I_n$ is a finite set, and another example of a (necessarily non-Noetherian) divisorial filtration of $\fm_R$-primary ideals such that the set of all  Rees valuations  is infinite.

\end{abstract}


\maketitle

\section{Introduction}

Let $(R,\fm_R,\kk)$ be a $d$-dimensional Noetherian local ring. In \cite[\S 2]{CM} the authors of this paper developed an intersection product $(-)_R$ on schemes that are proper and birational over $R$, see \autoref{subs:int_prod}. In \cite{CM}  this intersection product is used to provide geometric formulas for Hilbert-Samuel multiplicities and degree functions of ideals.   The goal of this paper is to extend these formulas to the setting of graded families of $\fm_R$-primary ideals, with a special focus on divisorial filtrations. We refer the reader to \cite[\S 2]{CM} for the notation on cycles and more details about the results on the intersection product that we use in this paper. This notation is consistent with that of \cite{Thor} and \cite{Fl}. 

Let $e(I)$ denote the (Hilbert-Samuel) multiplicity of  an $\fm_R$-primary ideal $I$. In \cite{Re},\cite{Re2}  Rees proved a celebrated theorem on the behavior of multiplicities under hyperplane sections. 
Assume $R$ is an analytically unramified  local  domain. The {\it degree function} $d_I: R\setminus\{0\}\to \NN$ is given by $d_I(x)=e(I(R/xR))$ for $x\ne 0$. In \cite[\S 2]{Re}, \cite[Theorem 9.31]{Re2} it is proved that $d_I$ can be expressed as
\begin{equation}\label{eq:Rees_Thmmm}
	d_I(x)=\sum_{i=1}^rd_i(I)v_i(x),
\end{equation}
where $v_1,\ldots, v_r$ are the Rees valuations of $I$ and $d_i(I)\in \ZZ_{\gs 0}$ for all $i$; Rees attributes to  Samuel   the beginning of the theory of degree functions \cite{Sam}.

In \cite{CM}, the authors prove the following theorem about multiplicities and degree functions. 
\begin{theorem}[{\cite[Theorem A, Corollary C]{CM}}]\label{TheoremCMAC}
	Let $(R,\fm_R,\kk)$ be a $d$-dimensional Noetherian local ring  and $I$ be an $\m_R$-primary ideal.
	\begin{enumerate}
		\item[{\rm (1)}] Let $\pi:Y\rightarrow \Spec(R)$ be a birational proper morphism such that $I\mathcal O_Y$ is invertible. Then 
$$
			e(I)=-((I\mathcal O_Y)^d)_R.
$$
\item[{\rm (2)}] Moreover, if $R$ is an analytically unramified  domain  and $\pi:Y=\Proj(\oplus_{n\gs 0}\,\overline{I^n})\rightarrow\Spec(R)$  is the natural projective morphism with  $\mathcal L_Y :=I\mathcal O_Y$, then for any $0\ne x\in R$ we have
$$
e(I(R/xR))=\sum_{E} v_E(x)(\mathcal L_Y^{d-1}\cdot E)_{R},
$$
where  $E$ ranges over the  integral components of $\pi^{-1}(\fm_R)$.   The local rings $\mathcal O_{Y,E}$ are discrete valuation rings and the canonical valuations $v_E$  are the Rees valuations of $I$. 
	\end{enumerate}
\end{theorem}

\autoref{TheoremCMAC}(1) is a generalization of the famous formula by  Ramanujam \cite{Ra}, which has also appeared  with more assumptions on $R$ in \cite{Ra}, \cite{Laz1}, and \cite{BLQ}.
Related formulas are also found in \cite{KT}, \cite{KV}, \cite{GGV}, \cite{CR}, \cite{CRPU}, and \cite{P}. 

\autoref{TheoremCMAC}(2) is a geometric interpretation of  degree functions.
 The coefficients $d_i(I)$ in  \autoref{eq:Rees_Thmmm} 
are uniquely determined by \cite[Theorem 9.42]{Re2}, so we have that these coefficients are equal to 
\begin{equation}\label{I6'}
	d_i(I)=((I\mathcal O_Y)^{d-1}\cdot E_i)_R,
\end{equation}
where $v_i$  are the Rees valuations of $I$.

If $R$ is a local ring, we define $\Div(\fm_R)$ to be the set of equivalence classes of $\fm_R$-valuations of $R$, where two $\fm_R$-valuations are equivalent if they have the same valuation ring; $\fm_R$-valuations are defined in \autoref{SecDiv}. For $v\in \Div(\fm_R)$, let $C(W,v)$ be the center of $v$ on $W$, defined in \autoref{SubSecDeg}. The following extension of \autoref{TheoremCMAC}(2)  is proven in \autoref{SecAnRam}.

\begin{theorem}[{\autoref{Bupval}}]\label{Bupval''} Let $(R,\fm_R,\kk)$ be a  $d$-dimensional analytically unramified local domain and $I\subset R$ be an $\fm_R$-primary ideal. Further assume  that $\phi:W\rightarrow \Spec(R)$ is a proper birational morphism such that 
	$W$ is normal, and $\L_W= I\mathcal O_W$ is invertible. Then  for any $0\ne x\in R$,
	$$
	e(I(R/xR))=\sum_{v\in \Div(\fm_R)}v(x)(\L_W^{d-1}\cdot C(W,v))_{R},
	$$
	where the sum is over  $v\in \Div(\fm_R)$ of $R$.  
	
	Furthermore, 
	for $v\in \Div(\fm_R)$, $(\mathcal L_W^{d-1}\cdot C(W,v))_R>0$ if and only if $v$ is a Rees valuation of $I$.
\end{theorem}

We now turn to multiplicities and degree functions of graded families of $\fm_R$-primary ideals, the primary focus of this article; graded families of $\fm_R$-primary ideals are defined in \autoref{SecDiv}. The multiplicity of  a graded family of $\fm_R$-primary ideals $\II=\{I_n\}_{n\gs 0}$ is defined as 
$$
e(\II)=\lim_{n\rightarrow\infty} \frac{d!\ell_R(R/I_n)}{n^d}.
$$
In \cite[Theorem 1.1]{C2} it is shown that this limit exists 
for all  graded families of $\fm_R$-primary ideals in $R$ if and only if $\dim\left( N(\hat R)\right)<d$, where $N(\hat R)$ is 
the  niladical of the $\fm_R$-adic completion of $R$. The existence of this limit with more restrictions on $R$ is shown in \cite{ELS} and \cite{LM}. 
We prove the following extension of \autoref{TheoremCMAC}(1)  in \autoref{DegGr}.
\begin{headthm}[{\autoref{Theorem5}}]\label{Theorem5''} 
	Let $(R,\fm_R,\kk)$ be a $d$-dimensional local ring such that $\dim\left( N(\hat R)\right)<d$. 
	Let $\II:=\{I_n\}_{n\gs 0}$ be a graded family of $\fm_R$-primary ideals and let $Y_n\rightarrow \Spec(R)$ be  proper birational morphisms such that $I_n\mathcal O_{Y_n}$ is invertible. 
	Then
	$$
	e(\II)=
	\lim_{n\rightarrow\infty}-\frac{((I_n\mathcal O_{Y_n})^d)_R}{n^d}.
	$$
\end{headthm}

As a result of \autoref{TheoremCMAC}(1), \autoref{Bupval''} and \autoref{Theorem5''}, we prove in  \autoref{DegGr} the following result on degree functions of graded families of $\fm_R$-primary ideals. 

\begin{headthm}[{\autoref{Theorem4})\label{Theorem4''} }]
Let $(R,\fm_R,\kk)$ be a $d$-dimensional  excellent normal   local ring.  
 Let $\II:=\{I_n\}_{n\gs 0}$ be a graded family of $\fm_R$-primary ideals.
Let  $0\neq x\in  R$ be such that $R/xR$ is reduced. For $n\gs 1$, let $Y_n$ be a proper birational morphism such that $Y_n$ is normal and $I_n\mathcal O_{Y_n}$ is invertible. 
Denote by $\II(R/xR)=\{I_n(R/xR)\}_{n\gs 0}$ the image of $\II$ in $R/xR$, which is a graded family of $\m_{R/xR}$-primary ideals in $R/xR$.  
 Then
	$$
	e(\II(R/xR))=\lim_{n\rightarrow \infty}\left(\sum_{v\in {\rm Div}(\m_R)}v(x)\frac{((I_n\mathcal O_{Y_n})^{d-1}\cdot C(Y_n,v))_R}{n^{d-1}}\right).
	$$
\end{headthm}

We have that $\frac{((I_n\mathcal O_{Y_n})^{d-1}\cdot C(Y_n,v))_R}{n^{d-1}}>0$ if and only if $v$ is a Rees valuation of $I_n$ by \autoref{Bupval''}. 
In \autoref{ExS0}, we construct a filtration  on a 2-dimensional regular local ring such that the sum and limit in the conclusions of \autoref{Theorem4''} do not commute. 
The following corollary shows that the sum and limit in \autoref{Theorem4''} does commute under some conditions.

\begin{headcor}[{\autoref{CorSwitch}}]\label{CorSwitch''} 
Under the notations in \autoref{Theorem4''}, further assume that the union of the sets of Rees valuations  of all the ideals $I_n$ is finite and that the limit 
	$$
	\lim_{n\rightarrow\infty}\frac{((I_n\mathcal O_{Y_n})^{d-1}\cdot C(Y_n,v))_R}{n^{d-1}}
	$$
	exists for all  $v\in \Div(\m_R)$. Then
	$$
	e(\II(R/xR))=\sum_{v\in \Div(\m_R)}v(x)\left(\lim_{n\rightarrow\infty}\frac{((I_n\mathcal O_{Y_n})^{d-1}\cdot C(Y_n,v))_R}{n^{d-1}}\right).
	$$
That is, the limit and sum in  \autoref{Theorem4''} commute.
\end{headcor}
If $\II=\{I_n\}_{n\gs 0}$ is Noetherian, i.e. $R[\II]=\oplus_{n\gs 0}\,I_n$ is  a finitely generated $R$-algebra, then the set of all Rees valuations of the ideals $I_n$ is a finite set, see \autoref{prop:Rees_finite}. However, if $\II$ is not Noetherian, then the set of all Rees valuations can be either a finite or an  infinite set. We give examples of divisorial filtrations of $\fm_R$-primary ideals on a 2-dimensional normal local ring which is essentially of finite type over a field in  \autoref{ExS1} and \autoref{ExS2} illustrating this behavior. An example of a symbolic filtration $\{P^{(n)}\}_{n\gs 0}$ of a height one prime ideal $P$ in a 2-dimensional normal local ring such that the set of all the Rees valuations of the symbolic powers $P^{(n)}$ is infinite is given in \cite{C10}.

  \autoref{CorSwitch''} extends the degree functions in \cite{Re}, as defined  right before  \autoref{eq:Rees_Thmmm}, to graded families of $\fm_R$-primary ideals under a suitable assumption. In this extension, we can think of the limits $d(\II,v):=\lim_{n\rightarrow\infty}\frac{((I_n\mathcal O_{Y_n})^{d-1}\cdot C(Y_n,v))_R}{n^{d-1}}$ as corresponding to  coefficients of $v(x)$ in \cite[\S 3]{ReSh} or \autoref{eq:Rees_Thmmm}.  
 \autoref{Theorem4''} and \autoref{CorSwitch''} raise the following general question. 

 \begin{question}\label{quest:limit_exists}
 Under the notations in \autoref{Theorem4''}, when does the limit
 $$
\lim_{n\rightarrow\infty}\frac{((I_n\mathcal O_{Y_n})^{d-1}\cdot C(Y_n,v))}{n^{d-1}}
$$
exists for $v\in \Div(\m_R)$?
 \end{question}
 
 Our next theorem shows that if $(R,\fm_R,\kk)$ 
 is a normal and excellent 
 $2$-dimensional local ring and $\II$ is a $\QQ$-divisorial filtration of $\fm_R$-primary ideals, then the limit in \autoref{quest:limit_exists} not only exists, but is given by a simple formula. For the statements and proof of the theorem we need some prior notation.

Let assumptions be as in the previous paragraph.  There exists a birational projective morphism $\pi:X\rightarrow \Spec(R)$ such that $X$ is nonsingular and $\pi^{-1}(\m_R)$ is a simple normal crossings 
divisor on $X$, so that all integral components of  $\pi^{-1}(\m_R)$ are nonsingular projective curves over $\kk$. Moreover, there exists an effective  $\QQ$-divisor $\Delta$ on $X$ with exceptional support,  such that $\Gamma(X,\mathcal O_X(-\lceil n\Delta\rceil))=I_n$ for $n\gs 0$ and $-\Delta$ is a nef divisor on $X$. The existence of $X$ and $\Delta$ with these properties is shown in \cite{CNag}.
We prove the following theorem in \autoref{Sec2dim}.

\begin{headthm}[{\autoref{TheoremLim}}]\label{TheoremLim''}  Let $R$ be a normal and excellent 2-dimensional local ring and $\II=\{I_n\}_{n\gs 0}$ be a $\QQ$-divisorial filtration of $\fm_R$-primary ideals on $R$.
Following the notation in the previous paragraph, let  $Y_n\rightarrow \Spec(R)$ be  projective birational morphisms such that $Y_n$ is normal and $I_n\mathcal O_{Y_n}$ is invertible. Then  
for every  $v\in \Div(\m_R)$ we have 
	$$
	\lim_{n\rightarrow\infty}\frac{(I_n\mathcal O_{Y_n}\cdot C(Y_n,v))}{n}=(-\Delta\cdot C(X,v)),
	$$
where $(\Delta\cdot C(X,v))$ is defined to be $\frac{1}{r}(r\Delta\cdot C(X,v))$ with $r\in \ZZ_{>0}$ and  $r\Delta$ an integral divisor.     
	In particular, this limit exists.
\end{headthm}

Finally, from \autoref{TheoremLim''} and \autoref{CorSwitch''} we conclude the following corollary.

\begin{headcor}[{\autoref{dim2lim}}]\label{dim2lim''} Under the notations in \autoref{TheoremLim''}, further assume that the union of the sets of Rees valuations of all the ideals $I_n$ is finite. For every $0\neq x\in  R$ with $R/xR$ is reduced we have
	$$
	e(\II(R/xR))=\sum_{v\in\Div(\m_R)}v(x)(-\Delta\cdot C(X,v)).
	$$
\end{headcor}


\section{Preliminaries}
In this section we provide the background information and notation that is used in the rest of the paper. Throughout we assume the ring $R$ is Noetherian.

\subsection{Graded families and divisorial filtrations}\label{SecDiv}
Let $R$ be a ring. A sequence of ideals $\II=\{I_n\}_{n\gs 0}$ is a {\it graded family of ideals} if $I_0=R$ and $I_nI_m\subset I_{n+m}$ for every $n,m\gs 0$. We say that that a graded family $\II$ is a {\it filtration} if in addition one has $I_{n+1}\subset I_n$ for every $n\gs 0$. We say that a graded family is {\it Noetherian} if its {\it Rees algebra} $R[\II]=\oplus_{n\gs 0}\,I_n$ is Noetherian.

We write $(\mathcal O_v,\fm_v,\kk_v)$ for the valuation ring, maximal ideal, and residue field of a valuation $v$.

An $R$-valuation of a local domain $R$ is a valuation $v$ of the quotient field of $R$ which is nonnegative on $R$. More generally, if $R$ is a local ring, then an $R$-valuation is an $R/P$-valuation $v$ of the quotient field of $R/P$ for some prime ideal $P$ of $R$. For $x\in R$, we write $v(x)=v(\overline{x})$, where $\overline{x}$ is the image of $x$ in $R/P$ if $x\not\in P$, and write $v(x)=\infty$ if $x\in P$. 

A {\it divisorial valuation} $v$ of a local domain $R$ is an $R$-valuation such that 
$$
\trdeg_{\kk(P)}(\kk_v)=\height(P)-1
$$ 
where $P=\fm_v\cap R$ ($P$ is the center of $v$ on $R$) and $\kk(P)$ is the quotient field of $R/P$ (\cite[Definition 9.3.1]{HS}).

If $v$ is a divisorial valuation, then there is equality in Abhyankar's inequality, so by \cite[Theorem 1]{Ab}, $\kk_v$ is a finitely generated field extension of  $\kk$.

In \cite[Theorem 9.3.2]{HS} it is shown that  the valuation ring $\mathcal O_v$ of  a divisorial valuation $v$ of a local domain $R$ is Noetherian. In particular, every divisorial valuation is $\ZZ$-valued by \cite[Corollary 6.5]{HS}. 
Further, if $R$ is locally 
analytically unramified, then, by \cite[Theorem 9.3.2]{HS}, every divisorial valuation $v$ of  $R$ is essentially of finite type over $R$.

 We now give the definition of an $\fm_R$-valuation, which can be found in \cite[\S 3]{ReSh}. 
Let $(R,\fm_R,\kk)$ be a $d$-dimensional local domain.  An {\it $\fm_R$-valuation} of $R$ is a valuation $v$ of $R$ centered at $\fm_R$ such that $v$ is $\ZZ$-valued and   $\kk_v$ is a finitely generated extension field of $\kk$ of transcendence degree $d-1$. Equivalently, $v$ is a divisorial valuation  with respect to $R$ that is centered at $\fm_R$ (as follows from the above discussion).

More generally, if  $R$ is a $d$-dimensional local ring, then an $\fm_R$-valuation of $R$ is an $\fm_R$-valuation of $R/P$, where $P$ is a minimal prime of $R$ such that $\dim(R/P)=d$. 
We define $\Div(\fm_R)$ to be the set of equivalence classes of $\fm_R$-valuations of $R$, where two $\fm_R$-valuations are equivalent if they have the same valuation ring.

Let $R$ be a ring and $v$ be an $R$-valuation whose valuation group is  a subgroup of  $\RR$, the ordered field of real numbers. For every $\lambda \in \RR$, we define the valuation ideal $I(v)_{\lambda}$ of $R$ as
$$
I(v)_{\lambda}=\{f\in R\mid v(f)\gs \lambda\}.
$$

For the remainder of this subsection, we assume that $R$ is an excellent normal local domain.
A {\it divisorial filtration} of $R$ is a  filtration $\II=\{I_n\}_{n\gs 0}$ such that there exist divisorial valuations $v_1,\ldots,v_r$ of $R$ and $\lambda_1,\ldots,\lambda_r\in \RR_{\gs 0}$ such that 
$$
I_n=I(v_1)_{n\lambda_1}\cap\cdots\cap I(v_r)_{n\lambda_r}
$$
for all $n\gs 0$. The filtration $\II$ is called a $\QQ$-divisorial filtration if all $\lambda_i\in \QQ$ and $\II$ is called a $\ZZ$-divisorial filtration if all $\lambda_i\in \ZZ$.

\begin{proposition}[\hspace{-.1pt}{\cite[\S 6]{CP}}]\label{PropRep}
Suppose that  $R$ is an excellent normal local ring and $v_1, . . . , v_r$ are divisorial valuations of
$R$. Then  there exists
an ideal $K$ of R such that its blowup $\phi:X\rightarrow \Spec(R)$ is normal and such that there 
exist prime Weil divisors $E_1, \ldots , E_r$ on $X$ with
$\mathcal O_{v_i}=\mathcal O_{X,E_i}$ for $1\ls i \ls r$.
\end{proposition}

Let $\II=\{I_n\}_{n\gs 0}$ be a divisorial valuation defined by $v_1,\ldots, v_r$ and $\lambda_1,\ldots,\lambda_r$. Let $X,E_1,\ldots, E_r$, be as in \autoref{PropRep}. Let $D$ be the Weil $\RR$-divisor $D=\sum \lambda_iE_i$. By the definition of the sheaf  $\mathcal O_X(-\lceil nD\rceil)$,
we have 
$$
\Gamma(X,\mathcal O_X(-\lceil nD\rceil))=I(v_1)_{n\lambda_1}\cap\cdots\cap I(v_r)_{n\lambda_r}=I_n
$$
for every $n\gs 0$.  Therefore, the $R$-scheme $X$ and the Weil divisor $D$ completely determine  $\II$. For this reason, we sometimes write $\II$ as  $\II(D)=\{I(D)_n\}_{n\gs 0}$, and call $\II(D)$ a {\it representation} of $\II$.

 \subsection{Intersection theory over local rings}\label{subs:int_prod}

In this subsection, we give a quick outline of the intersection theory that is used in this paper. This theory is defined and developed in \cite{CM}. The reader is referred to \cite[\S 2]{CM} for details and proofs.

Thorup \cite{Thor} has developed a theory of rational equivalence for finite type schemes over a Noetherian scheme. 
He comments that most of Fulton's intersection theory over a field \cite{Fl} extends to this setting. 
 When $X$ is a finite type scheme over a field, 
Fulton \cite{Fl} defines rational equivalence and makes cycles modulo rational equivalence on $X$ into a graded ring by using  ordinary dimension. When $X$ is over a Noetherian base, Thorup \cite{Thor}  uses  relative dimension  to define   rational equivalence on $X$ and to make cycles modulo rational equivalence  a graded ring. 
 In Thorup's theory of rational equivalence, it  is necessary to throw out components of intersections which have smaller than expected dimension if the base is not equidimensional or not universally catenary.

Let $(R,\fm_R,\kk)$ be a $d$-dimensional Noetherian local ring and $\pi:X\rightarrow S:=\Spec(R)$ be a finite type morphism. For an integral subscheme $V$ of $X$ with image $T=\pi(V)$ in $S$, the {\it relative dimension} of $V$ over $S$ is
$$
\dim_SV=\trdeg(R(V)/R(T))-\dim(\mathcal O_{S,T}),
$$
here $R(V)$ denotes the function field of $V$, i.e., $R(V) = O_{X,V} /\fm_V$ where $O_{X,V}$ is the local ring of the generic point of $V$ and $\fm_V$ is its maximal ideal.
Assume now that  $\pi:X\rightarrow \Spec(R)$ is a proper birational morphism. Let $\alpha$ be a cycle of relative dimension $k$ on $X$ and $F_1,\ldots,F_t$ be Cartier divisors on $X$ with $t\gs k+d$ and at least one of the $F_i$ or $\alpha$ is supported above $\m_R$. Then  define an intersection number
$$
(F_1\cdot\ldots\cdot F_t\cdot \alpha)_R=\int_{\kk}F_1\cdot\ldots\cdot F_t\cdot\alpha,
$$
where the latter is  the degree of the cycle 
$$
F_1\cdot\ldots\cdot F_t\cdot\alpha=\sum n_ip_i\in A_{-d}(X/\Spec(R))=
A_0(X_{\kk}/\kk);
$$
here the $p_i$ are closed points of $X_{\kk}$ and $n_i\in \ZZ$.  Then the intersection number is the degree 
$$
\int_{\kk}F_1\cdot\ldots\cdot F_t\cdot\alpha=\sum n_i[R(p_i):\kk].
$$
We observe that the product 
$
(F_1\cdot,\ldots,\cdot F_t\cdot \alpha)_R$ 
can be  explicitly calculated using the intersection product of Snapper-Mumford-Kleiman \cite{Sn}, \cite{Mum}, \cite{Kl} for proper schemes over a field.

\subsection{Riemann-Roch theorems for curves}\label{SecRR}
In this subsection we summarize the famous Riemann-Roch Theorems for curves. 
A reference where these theorems are proven over an arbitrary field  is \cite[\S 7.3]{Lin}. The results that we need are stated in \cite[Remark 7.3.33]{Lin}.

Let $E$ be a regular integral 
projective curve over a field $\kk$. For $\mathcal F$ a coherent sheaf on $E$ define
$h^i(\mathcal F)=\dim_\kk \left(H^i(E,\mathcal F)\right)$ for every $i\gs 0$. 
Let $D=\sum_i m_ip_i$ be a (Weil) divisor on $E$, where $p_i$ are prime divisors on $E$ (closed points) and $a_i\in \ZZ$, and $\mathcal O_E
(D)$ its associated invertible sheaf. Define
$$
\deg(D)=\deg(\mathcal O_E(D))=\sum_i m_i[R(p_i):\kk].
$$
Here $R(p_i)=\mathcal O_{E_i,p_i}/\m_{p_i}$,
where $\m_{p_i}$ is the maximal ideal of $\mathcal O_{E_i,p_i}$. The {\it Riemann-Roch formula} is given by 
\begin{equation}\label{eq44}
	\chi(\mathcal O_{E}(D)):=h^0(\mathcal O_{E}(D))-h^1(\mathcal O_{E}(D))=\deg(D)+1-p_a(E),
\end{equation}
where $p_a(E)$ is the  arithmetic genus of $E$. 

If $\mathcal L$ is an invertible sheaf on $E$, then $\mathcal L\cong \mathcal O_E(D)$ for some divisor $D$ on $E$, and we may define $\deg(\mathcal L)=\deg(\mathcal O_X(D))=\deg(D)$.  

By Serre duality 
$H^1(E,\mathcal O_{E}(D))\cong H^0(E,\mathcal O_{E}(K_E-D))$,  
where $K_E$ is a canonical divisor on $E$. As a consequence, by \cite[Proposition 3.25(b)]{Lin} we have that the inequality
\begin{equation*}\label{eq43}
	\deg(D)>2p_a(E)-2=\deg(K_E)\,\,\mbox{ implies  }\,\,H^1(E,\mathcal O_{E}(D))=0.
\end{equation*}





We will apply the above formulas in the case that $E$ is a prime exceptional divisor, i.e., $E\subset\pi^{-1}(\fm_R)$, 
for a resolution of singularities $\pi:X\rightarrow\Spec(R)$ with $(R,\fm_R, \kk)$ a 2-dimensional local ring. We have that  $E$ is a projective nonsingular (by assumption) integral curve over $\kk$, 
and if   $D$ is a Cartier   divisor on $X$,  then $\deg(\O_X(D)\otimes \O_E)=(D\cdot E)_R$, where $(-)_R$ 
is the product developed in \cite{CM} and reviewed in \autoref{subs:int_prod}. To see this, let $j:E\rightarrow X$ be the inclusion. Then 
$$
D\cdot[E]=[j^*D]^1\in A_{-1}(X_{\kk}/\Spec(R))=A_0(X_{\kk}/\kk)
$$
by   \cite[formulas (6) and (9)]{CM}. Thus
$$
(D\cdot E)_R=\int_{X_{\kk}}D\cdot [E]=\int_{X_{\kk}}[j^*D]^1=\deg(\mathcal O_X(D)\otimes\mathcal O_E)
$$
by  \cite[formulas (10) and (7) ]{CM}.


\subsection{Divisors of rational functions and of sections of invertible sheaves}\label{SecCart}

Let $X$ be an integral Noetherian scheme.  The fundamental cycle $[\divv(r)]\in Z(X)$ of $r\in R(X)^*$ is defined in \cite[Chapter 1]{Fl} and in \cite[Subsection 2.1]{CM}. In this paper we  denote this cycle by $\Div(r)$.

Let  $\mathcal L$ be an invertible sheaf on $X$. There exists an open cover $\{U_i\}$ of $X$ and for each $i$ a local generator $\tau_i$ of $\mathcal L|U_i$ such that $\mathcal L|U_i=\tau_i\mathcal O_{U_i}$. Let $0\ne \sigma\in\Gamma(X,\mathcal L)$. 
There exist nonzero  $f_i\in \Gamma(U_i,\mathcal O_X)$ such that $\sigma|U_i=\tau_if_i$ for every $i$ and  $f_i|U_i\cap U_j=f_j|U_i\cap U_j$ for every $i,j$. The data  $\{(f_i,U_i)\}$ is  an effective Cartier divisor on $X$ which is called the divisor of $\sigma$ and written as $\Div(\sigma)$. The support $\Supp(\Div(\sigma))$ of $\sigma$ is the 
 closed codimension one subscheme $Z_{\sigma}$ of $X$ such that the ideal sheaf of $Z_{\sigma}$ in $X$ is given by $\mathcal I_{Z_\sigma}|U_i=f_i\mathcal O_{X_i}$.

 The following proposition is used in the proof of \autoref{PropLim2}

 \begin{proposition}\label{LemInj} Let  $X$ be an integral scheme, $\mathcal L$ and $\mathcal M$ be invertible sheaves on $X$, and   $Y\subset X$ be a closed integral subscheme. Suppose that $0\ne \sigma\in \Gamma(X,\mathcal L)$. Then
 \begin{equation*}\label{LemInj1}
 \mathcal L^{-1}\otimes \M
 \xrightarrow{\sigma\otimes 1}
 \mathcal O_X\otimes\mathcal M\cong \mathcal M
 \end{equation*} 
 is injective. Moreover,  
  the induced map
\begin{equation*}\label{LemInj2}
\mathcal L^{-1}\otimes\mathcal M\otimes\mathcal O_Y
 \xrightarrow{\sigma\otimes1\otimes 1}
\mathcal O_X\otimes\mathcal M\otimes\mathcal O_Y\cong\mathcal M\otimes\mathcal O_Y
\end{equation*}
is injective if and only if $Y$ is not in the support of $\Div(\sigma)$. 
In particular, if $\mathcal L\cong \mathcal O_X(D)$, for some divisor $D$ on $X$, and $\sigma\in \Gamma(X,\mathcal L)$ is such that $\Div(\sigma)=D$ and $Y$ is an integral subscheme of $X$, then 
$$
\mathcal O_X(-D)\otimes\mathcal M\otimes\mathcal O_Y
\xrightarrow{\sigma\otimes 1\otimes 1}
\M\otimes\mathcal O_Y
$$
is injective if and only if $Y$ is not in the support of $D$.
\end{proposition}

 \begin{proof} With the notation given at the beginning of this subsection,  $\mathcal L^{-1}|U_i=\frac{1}{\tau_i}\mathcal O_{U_i}$, and so
 $$
 (\mathcal L^{-1}\otimes\mathcal M)|U_i
 \xrightarrow{\sigma\otimes 1}
 (\mathcal O_X\otimes\mathcal M)|U_i
 $$
 is the map 
 $
 \frac{h}{\tau_i}\otimes a\mapsto f_ih\otimes a
 $
 for $h\in\mathcal O_{U_i}$. This map is injective since $f_i\ne 0$. 
 
 The map
 $$
 (\mathcal L^{-1}\otimes\mathcal M\otimes\mathcal O_Y)|U_i
 \xrightarrow{\sigma\otimes1\otimes 1}
(\mathcal O_X\otimes\mathcal M\otimes\mathcal O_Y)|U_i
$$
is then given by
$
\frac{h}{\tau_i}\otimes a\otimes b\mapsto f_ih\otimes a\otimes b=1\otimes a\otimes \overline{f_ih} b
$, 
where $\overline{f_ih}$ is the class of $f_ih$ in $\mathcal O_Y|U_i$. Now $\overline{f_ih}=0$ if and only if $Y$ is not in the support of $\Div(\sigma)$ since $Y$ is integral, and the conclusion follows. 
 \end{proof}

\section{Degree functions of graded families of ideals} \label{SecAnRam}

\subsection{Degree functions on analytically unramified local domains}\label{SubSecDeg}
In this subsection, let $R$ be a $d$-dimensional analytically unramified local domain and let
$I$ be an $\fm_R$-primary ideal. Let $\pi:Y:=\Proj(\oplus_{n\gs 0}\,\overline{I^n})\rightarrow\Spec(R)$ be the natural morphism. By \cite[Corollary 9.21]{HS}, $\oplus_{n\gs 0}\,\overline{I^n}$ is a finitely generated $R$-algebra. By \cite[Proposition 4.9]{CM}, $Y$ is a projective $R$-scheme and the Rees valuations of $I$ are the 
canonical valuations $v_E$ of the valuation rings $\mathcal O_{Y,E}$, where $E$ are the integral components of $\pi^{-1}(\fm_R)$.

Let $\mathcal L_Y=I\mathcal O_Y$. Since $\mathcal L_Y$ is ample on $Y$,
$(\mathcal L_Y^{d-1}\cdot E)_R>0$ if and only if $\dim(E)=d-1$, which holds if and only if $v_E\in \Div(\fm_R)$, since $\dim (E)=\trdeg_{\kk}(R(E))$ as $E$ is a projective variety over $\kk$.

Let $\phi:W\rightarrow \Spec(R)$ be proper and birational. 
For $v\in \Div(\fm_R)$, let $C(W,v)$ be the center of $v$ on $W$, that is,  the integral subscheme  of $W$ that is the closure in $W$ of the unique point $q$ of $W$ such that the local ring $\mathcal O_{W,q}$ is dominated by the valuation ring  of $v$.

\begin{lemma}\label{Rval} For $v\in \Div(\fm_R)$, $(\mathcal L_Y^{d-1}\cdot C(Y,v))_R>0$ if and only if $v$ is a Rees valuation of $I$. 
\end{lemma}

\begin{proof} If $v\in \Div(\fm_R)$ is a Rees valuation of $I$, then  $(\mathcal L_Y^{d-1}\cdot C(Y,v))_R>0$ as shown above.
Suppose that $v\in \Div(\fm_R)$ is such that $(\mathcal L_Y^{d-1}\cdot C(Y,v))_R>0$. Then $E=C(Y,v)$ is an integral subscheme of $\pi^{-1}(\fm_R)$ of dimension $d-1$ such that $\mathcal O_{Y,E}$ is dominated by $\mathcal O_v$. Since $\mathcal O_{Y,E}$ is a one dimensional valuation ring we have that $\mathcal O_v=\mathcal O_{Y,E}$ and so $v$ is equivalent to $v_E$.
\end{proof}

The following theorem is a version of a theorem of Rees on degree functions \cite{Re}, as  explained in the introduction.

\begin{theorem}\label{Bupval} Let $(R,\fm_R,\kk)$ be a  $d$-dimensional 
 analytically unramified local domain and $I\subset R$ be an $\fm_R$-primary ideal.  Then $\oplus_{n\gs 0}\,\overline{I^n}$ is a finitely generated $R$-algebra, $Y=\Proj(\oplus_{n\gs 0}\,\overline{I^n})$ is a projective $R$-scheme and $I\mathcal O_Y$ is invertible. 
Assume  that $\phi:W\rightarrow \Spec(R)$ is a proper birational morphism such that 
 $W$ is normal, and $\L_W= I\mathcal O_W$ is invertible. Then  for any $0\ne x\in R$,
\begin{equation}\label{eqF1}
e(I(R/xR))=\sum_{v\in \Div(\fm_R)}v(x)(\L_W^{d-1}\cdot C(W,v))_{R}.
 \end{equation}
Furthermore ,
for $v\in \Div(\fm_R)$, $(\mathcal L_W^{d-1}\cdot C(W,v))_R>0$ if and only if $v$ is a Rees valuation of $I$.
\end{theorem}

\begin{proof} 
 The statement about $Y$ follows from the first paragraph of this subsection and \eqref{eqF1} is immediate from \cite[Corollary 4.16]{CM}.

Since  $I\mathcal O_W$ is invertible, there is a factorization $W\rightarrow \Proj(\oplus_{n\gs 0}\,I^n)$, and since $W$ is normal, there is then a factorization $\lambda:W\rightarrow Y$. Thus by the projection formula \mbox{\cite[Proposition 2.6]{CM}}, for $v\in \Div(\fm_R)$, 
$$
(\mathcal L_W^{d-1}\cdot C(W,v))_R=((\lambda^*\mathcal L_Y)^{d-1}\cdot C(W,v))_R
=(\mathcal L_Y^{d-1}\cdot \lambda_* C(W,v))_R.
$$
The image of $C(W,v)$ in $Y$ is $\lambda(C(W,v))=C(Y,v)$. If $\dim (C(Y,v))=\dim (C(W,v))$, then the index
$\mu=[R(C(W,v)):R(C(Y,v)]$ is finite and positive. 
As defined in \cite[\S 1.4]{Fl} or \cite[\S 2.1]{CM}, 
$$
\lambda_*C(W,v)=\left\{\begin{array}{ll} 0 & \mbox{ if }\dim (C(Y,v))<\dim (C(W,v))\\
\mu C(Y,v) & \mbox{ if } \dim (C(Y,v))=\dim (C(W,v)).
\end{array}\right.
$$
Thus
\begin{equation}\label{eqG}
(\mathcal L_W^{d-1}\cdot C(W,v))_R=\left\{ \begin{array}{ll}
0& \mbox{if } \dim (C(Y,v))<\dim (C(W,v))\\
\mu(\mathcal L_Y^{d-1}\cdot C(Y,v))_R&\mbox{if }\dim (C(Y,v))=\dim (C(W,v)).
\end{array}\right.
\end{equation}
Suppose that $v$ is a Rees valuation. Then $\dim (C(V,v))=d-1$ by the proof of \autoref{Rval}. Thus $\dim (C(W,v))=d-1$ since $\lambda(C(W,v)) =C(Y,v)$ and $\lambda$ is birational.
 Thus $$(\mathcal L_W^{d-1}\cdot C(W,v))_R=\mu(\mathcal L_Y^{d-1}\cdot C(Y,v))_R>0$$ by \autoref{eqG} and \autoref{Rval}.

Suppose that $(\mathcal L_W^{d-1}\cdot C(W,v))_R>0$. Then $\dim (C(W,v))=\dim (C(Y,v))$ by \autoref{eqG}, and so 
$$\mu(\mathcal L_Y^{d-1}\cdot C(Y,v))_R=(\mathcal L_W^{d-1}\cdot C(W,v))_R>0,
$$
which implies that $v$ is a Rees valuation by \autoref{Rval}.
\end{proof}


\subsection{Degree functions of graded families}\label{DegGr}

The next theorem expresses  the multiplicity of a graded family in terms of the intersection products of its ideal members.  Here $N(S)$ denotes the nilradical of the ring $S$.

\begin{theorem}\label{Theorem5} 
Let $(R,\fm_R,\kk)$ be a $d$-dimensional local ring such that $\dim (N(\hat R))<d$. 
 Let $\II:=\{I_n\}_{n\gs 0}$ be a graded family of $\fm_R$-primary ideals and let $Y_n\rightarrow \Spec(R)$ be  proper birational morphisms such that $I_n\mathcal O_{Y_n}$ is invertible. 
    Then
	$$
	e(\II)=
	\lim_{n\rightarrow\infty}-\frac{((I_n\mathcal O_{Y_n})^d)_R}{n^d}.
	$$
\end{theorem}

\begin{proof} By the Volume = Multiplicity formula  of  \mbox{\cite[Theorem 1.1]{Cvol}},
	$
	e(\II)=\lim_{n\rightarrow\infty}\frac{e(I_n)}{n^d}.
	$
	The conclusion now follows from \autoref{TheoremCMAC}(1).
\end{proof}

The following theorem follows from \autoref{Theorem5} and   \cite[Theorem 8.3]{C3}.

\begin{proposition}\label{rem:anti_pos} Let$(R,\fm_R,\kk)$ be a normal local ring which is essentially of finite type over a field. 
Let $\II=\{I(nD)\}_{n\gs 0}$ be a divisorial filtration of $\m_R$-primary ideals. Then
 $$
e(\II)
=\lim_{n\rightarrow\infty}-\frac{((I(nD)\mathcal O_{Y_n})^d)_R}{n^d}
=-\langle(-D)^d\rangle,
$$
where the latter is the anti-positive intersection product of $D$ defined in  \cite[\S 7]{C3}.
\end{proposition}

In the next theorem we analyze the behavior of the multiplicity of a graded family modulo a radical principal ideal. We observe that in our setting of Noetherian normal domain such a principal ideal always exist, see \autoref{rem:rad_princ_ideal}. 

\begin{theorem}\label{Theorem4} 
Let $(R,\fm_R,\kk)$ be a $d$-dimensional  excellent normal   local ring.  
 Let $\II:=\{I_n\}_{n\gs 0}$ be a graded family of $\fm_R$-primary ideals.
Let  $0\neq x\in  R$ be such that $R/xR$ is reduced. For $n\gs 1$, let $Y_n$ be a proper birational morphism such that $Y_n$ is normal and $I_n\mathcal O_{Y_n}$ is invertible. 
Denote by $\II(R/xR)=\{I_n(R/xR)\}_{n\gs 0}$ the image of $\II$ in $R/xR$, which is a graded family of $\m_{R/xR}$-primary ideals in $R/xR$.  
 Then
	$$
	e(\II(R/xR))=\lim_{n\rightarrow \infty}\left(\sum_{v\in {\rm Div}(\m_R)}v(x)\frac{((I_n\mathcal O_{Y_n})^{d-1}\cdot C(Y_n,v))_R}{n^{d-1}}\right).
	$$
\end{theorem}

\begin{proof} $R/xR$ is excellent and reduced, so $R/xR$ is analytically unramified by \cite[Scholie IV.7.8.3]{EGAIV}. Thus
$$
e(\II(R/xR))=\lim_{n\rightarrow\infty}\frac{e(I_n(R/xR))}{n^{d-1}}
=
\lim_{n\rightarrow \infty}\left(\sum_{v\in {\rm Div}(\m_R)}v(x)\frac{((I_n\mathcal O_{Y_n})^{d-1}\cdot C(Y_n,v))_R}{n^{d-1}}\right)
$$
by \cite[Theorem 6.5]{C2} and   \autoref{Bupval}.
	\end{proof}

\begin{remark}\label{rem:rad_princ_ideal} 
If $R$ is as in \autoref{Theorem4}, by \cite[Lemma 3.14]{Artin}  there always exists $0\neq x\in R$ such that $R/xR$ is reduced
  (see also
\cite[tag 0BWZ, Proposition 15.125.10]{Stack}).
\end{remark}

 \begin{remark} An excellent domain $R$ is formally equidimensional by \cite[Scholie IV.7.8.3(x)]{EGAIV}. Thus all Rees valuations of an $\fm_R$-primary ideal $I$ of $R$ are $\fm_R$-valuations by \cite[Remark 4.6]{CM}.
\end{remark}

\begin{corollary}\label{CorSwitch} 
Under the notations in \autoref{Theorem4}, further assume that the union of the sets of Rees valuations  of all the ideals $I_n$ is finite and that the limit 
	$$
	\lim_{n\rightarrow\infty}\frac{((I_n\mathcal O_{Y_n})^{d-1}\cdot C(Y_n,v))_R}{n^{d-1}}
	$$
	exists for all  $v\in \Div(\m_R)$. Then
	$$
	e(\II(R/xR))=\sum_{v\in \Div(\m_R)}v(x)\left(\lim_{n\rightarrow\infty}\frac{((I_n\mathcal O_{Y_n})^{d-1}\cdot C(Y_n,v))_R}{n^{d-1}}\right).
	$$
That is, the limit and sum in  \autoref{Theorem4} commute.
\end{corollary}

\begin{proof} 	Let $\{v_1,\ldots,v_r\}$ be the set of all the Rees valuations of the ideals  $I_n$. Suppose that $v\in\Div(\fm_R)$ is not equivalent to any of the $v_i$. Then $((I_n\mathcal O_{Y_n})^{d-1}\cdot C(Y_n,v))_R=0$ for all $n$ by \autoref{Bupval}, and so, by \autoref{Theorem4}, 
\begin{align*}
		e(\II(R/xR))
        =\sum_{i=1}^rv(x)\left(\lim_{n\rightarrow\infty}\frac{((I_n\mathcal O_{Y_n})^{d-1}\cdot C(Y_n,v))_R}{n^{d-1}}\right).
\end{align*}
\end{proof}

\begin{remark}\label{rem:degree_func_filtration}
 The previous result \autoref{CorSwitch} extends the degree functions in \cite{Re, ReSh, Re2} to graded families of $\fm_R$-primary ideals under a suitable assumption. In this extension, we can think of the limits $d(\II,v):=\lim_{n\rightarrow\infty}\frac{((I_n\mathcal O_{Y_n})^{d-1}\cdot C(Y_n,v))_R}{n^{d-1}}$ as corresponding to  coefficients of $v(x)$ in \cite[\S 3]{ReSh}. In \autoref{TheoremLim} we show these limits exist for divisorial filtrations in dimension 2.
\end{remark}

\section{Degree functions in 2-dimensional normal local rings}\label{Sec2dim}

In this section, we are concerned with resolutions of singularities of two dimensional excellent  local domains. The  existence of  resolution of singularities for two-dimensional reduced excellent schemes is proven in  \cite{L2} or \cite{CJS}.
For simplicity of notation, we   write $(-)$ for the intersection product $(-)_R$. In the setting of this section, the intersection product of divisors is classical; it is in fact the intersection product developed in \cite{L}, as is shown before the statement of \cite[Theorem 5.5]{CM}.

In the following example we show that in general the limit and sum in \autoref{Theorem4} do not commute. Recall that \autoref{CorSwitch} gives a sufficient condition  for the limit and sum to commute. 

\begin{example}\label{ExS0}  There exists a 
2-dimensional regular local ring $(R,\fm_R,\kk)$, which is the localization at a maximal ideal of a polynomial ring in two variables over an algebraically closed field,  and a  
filtration $\II=\{I_n\}_{n\gs 0}$ of $\m_R$-primary   such that:
	\begin{enumerate}
		\item[(1)] There exists an infinite sequence of inequivalent $\fm_R$-valuations $\{v_i\}_{i\gs 0}$ such that  $v_0,v_1,\ldots,v_n$ are the Rees valuations of $I_n$ for all $n$.
		\item[(2)] Let $R[\II]=\oplus_{n\gs 0}\,I_n$ be the  Rees algebra of $\II$. Then 
		$$
		\sqrt{\m_RR[\II]}=\m_R\oplus R[\II]
		_{>0},
		$$
		so that the {\it analytic spread} of $\II$ is zero, i.e.,  $\ell(\II)=\dim \left(R[\II]/\m_RR[\II]\right)=0$.
		\item[(3)] For any $0\neq x\in R$ with  $R/xR$  reduced,  the limit and sum in \autoref{Theorem4} do not commute.
	\end{enumerate}
\end{example}

\begin{proof}  Let $\kk$ be an algebraically closed field and  $R=\kk[a,b]_{(a,b)}$ be the localization of the polynomial ring $\kk[a,b]$ at the maximal ideal $(a,b)$. 
Let $\pi:X\rightarrow \Spec(R)$ be the blowup of $\m_R=(a,b)$ with exceptional divisor $E$. Let $\{q_i\}_{i\in \ZZ_{>0}}$ be an infinite set of distinct points on $E$. 

For $n\gs 1$, let $\pi_n:Y_n\rightarrow X$ be the blowup of $q_1,\ldots,q_n$, and for $m\gs n$ let $\pi_{m,n}:Y_m\rightarrow Y_n$ be the natural morphism such that $\pi_m= \pi_n\circ \pi_{m,n}$. 
	Let $E_0(n)$ be the strict transform of $E$ on $Y_n$,  and for $1\ls i\ls n$ let $E_i(n)$ be the exceptional curves of $\pi_n$ that contract to $q_i$. 
    To simplify notation, we  write $E_i$ for $E_i(n)$ when the context makes it clear which scheme $Y_n$ the curves are on.  Let $D_n$ be the divisor on $Y_n$ defined by 
	\begin{equation}\label{Ex4}
		D_n=\pi_n^*((2n+1)E)+E_1+\cdots+E_n=(2n+1)E_0+(2n+2)E_1+\cdots+(2n+2)E_n,
	\end{equation}
    where the second equality uses $\pi_n^*(E)=E_0+E_1+\cdots+E_n$.


	We have the following intersection numbers
    \begin{align}\label{EXIN}
	&(E_0^2)=-(n+1),\qquad
    (E_i^2)=-1\mbox{ for }i\gs 1,\\
    \nonumber &(E_i\cdot E_j)=0\mbox{ for $j>i\gs 1$,}\qquad \text{and}\qquad (E_i\cdot E_0)=1
	\mbox{ for }i\gs 1. 
	\end{align}
  
 We prove the formulas in \autoref{EXIN} by induction on $n$. By \autoref{TheoremCMAC}(1), 
$(E^2)=-e(\m_R)=-1$, since $R$ is a regular local ring. Thus \autoref{EXIN} is true for $n=0$ (taking $Y_0=X$). Suppose \autoref{EXIN} is true for some $n\gs 0$, we  prove that it is true for $n+1$. First, we have that $(E_{n+1}(n+1)^2)=e(\m_{q_{n+1}})=-1$ by \autoref{TheoremCMAC}(1), since $\mathcal O_{Y_{n+1},q_{n+1}}$ is a regular local ring with maximal ideal $\m_{q_{n+1}}$. 
Moreover, for $1\ls i\ls n$, $(E_i(n+1)\cdot E_{n+1}(n+1))=0$ since 
$$
E_i(n+1)\cdot E_{n+1}(n+1)\cdot [Y_{n+1}]=E_i(n+1)\cdot [E_{n+1}(n+1)]=0
$$
in $A_{-2}(Y_{n+1}/\Spec(R))$; the latter equality holds since $\mathcal O_{Y_{n+1}}(E_i(n+1))\otimes\mathcal O_{E_{n+1}(n+1)}\cong \mathcal O_{E_{n+1}(n+1)}$, as the supports of $E_i(n+1)$ and $E_{n+1}(n+1)$ are disjoint.

By a local calculation above the point $q_{n+1}$ blown up by the morphism $\pi_{n+1,n}:Y_{n+1}\rightarrow Y_n$, we see that $E_0(n+1)$ and $E_{n+1}(n+1)$ intersect in a single point $w$ 
 and that  local equations for $E_0(n+1)$ and $E_{n+1}(n+1)$ generate the maximal ideal $\fm_{w}$ in $\mathcal O_{Y_{n+1},w}$. Now 
$(E_0(n+1)\cdot E_{n+1}(n+1))=1$ since
$$
E_0(n+1)\cdot E_{n+1}(n+1)\cdot [Y_{n+1}]=E_0(n+1)\cdot [E_{n+1}(n+1)]=[w]
$$
in $A_{-2}(Y_{n+1}/\Spec(R))$ so 
$$
(E_0(n+1)\cdot E_0(n+1))=\int_{\kk}[w]=1.
$$

We note that 
 $\pi_{n+1,n}^*(E_0(n))=E_0(n+1)+E_{n+1}(n+1)$ and that
for $1\ls i\ls n$ the map $\pi_{n+1,n}$ is an isomorphism in a neighborhood of $E_i(n)$.  Therefore, by 
\cite[Proposition 2.8]{CM}
 and induction, we have 
$$
(E_0(n+1)\cdot E_i(n+1))
= (\pi_{n+1,n}^*(E_0(n))\cdot \pi_{n+1,n}^*(E_i(n)))
=(E_0(n)\cdot E_i(n))=1 \text{ for } 1\ls i\ls n,
$$
$$
(E_i(n+1)\cdot E_j(n+1))
= (\pi_{n+1,n}^*(E_i(n))\cdot \pi_{n+1,n}^*(E_j(n)))
=(E_i(n)\cdot E_j(n))=0 \text{ for } 1\ls i<j\ls n,
$$
$$
\text{and }(E_i(n+1)^2)
= (\pi_{n+1,n}^*(E_i(n))^2)
=(E_i(n)^2)=-1 \text{ for } 1\ls i\ls n.
$$
Likewise, 
$$
\big(\pi_{n+1,n}^*(E_0(n))\cdot \pi_{n+1,n}^*(E_0(n))\big)=(E_0(n)\cdot E_0(n))=-(n+1).
$$
On the other hand,
\begin{align*}
(\pi_{n+1,n}^*(E_0(n))\cdot \pi_{n+1,n}^*(E_0(n)))&=((E_0(n+1)+E_{n+1}(n+1))\cdot (E_0(n+1)+E_{n+1}(n+1)))\\
&=(E_0(n+1)^2)+2-1.
\end{align*}
Thus $(E_0(n+1)^2)=-(n+2)$.

	From  \autoref{Ex4} and \autoref{EXIN} we obtain $(-D_n\cdot E_0)=n+1$ and $(-D_n\cdot E_i)=1$ for $i\gs 1$. Thus 
    \cite[Theorem 12.1]{L} implies that $\mathcal O_{Y_n}(-D_n)$ is generated by global sections over $\Spec(R)$ and it  is very ample over $\Spec(R)$.  
	Let 
	$$
	I_n=\Gamma(Y_n,\mathcal O_{Y_n}(-D_n)).
	$$
	It follows that $I_n\mathcal O_{Y_n}=\mathcal O_{Y_n}(-D_n)$ and $Y_n$ is the blowup of $I_n$. 
	Let $v_0$ be the canonical valuation of $\mathcal O_{X,E}$, and $v_i$ be the canonical valuation of $\mathcal O_{Y_i,E_i}$ for $i\gs 1$. Then, by \autoref{Ex4}, 
	$$
	I_n=I(v_0)_{2n+1}\cap I(v_1)_{2n+2}\cap\cdots\cap I(v_n)_{2n+2}
    \quad
    \text{for}
    \quad
    n\gs 1.
	$$
    We conclude that $\II=\{I_n\}_{n\gs 0}$   is a  filtration of $R$. The ideal 
    $I_n$ is integrally closed, since it is an intersection of valuation ideals \cite[Lemma 4.6]{BDHM}. Moreover, since $R$ is a 2-dimensional regular local ring,  $I_n^m$ is an integrally closed ideal for all $m\gs 1$ by  \cite[Theorem 2',  p.385]{ZS2} or  \cite[Lemma 7.3]{L}. Thus $R[I_nt]$ is integrally closed in $R[t]$ \cite[Proposition 5.2.1]{HS}, and so the blowup of $I_n$, $Z_n=\Proj(R[I_nt])$, is a normal scheme. 
    By   \cite[Proposition II.7.14]{H}, we have a natural birational morphism $\phi:Y_n\rightarrow Z_n$ which is proper by  \cite[Corollary II.4.8]{H}. Moreover, $\phi$ is  quasi-finite since $I_n\mathcal O_{Y_n}=\phi^*\left(\mathcal O_{Z_n}(1)\right)$ is very ample, so that no curves of $Y_n$ are contracted to points on $Z_n$. 
    Thus $\phi$ is an isomorphism by Zariski's Main Theorem. In fact, $\phi$ is finite by \cite[Corollary 1.10]{Mil} and since $\phi$ is birational and $Z_n$ is normal, $\phi$ is an isomorphism.
Thus $Y_n$ is the blowup of $I_n$. Since $Y_n$ is normal and $I_n\mathcal O_{Y_n}
=\mathcal O_{Y_n}(-D_n)$, it follows from \autoref{Ex4} 
    that  $v_0,v_1,\ldots,v_n$ are the Rees valuations of $I_n$ for all $n\gs 1$.  This finishes the proof of part (1).
	
	We continue with the proof of part (2). By  \cite[Lemma 7.3]{L}, for  every $n,m\gs 0$ with $m\gs 3$, the image of the natural map from
	$$
	\m_R^{m-2}\otimes_RI_{mn}=\Gamma\left(Y_{mn},\mathcal O_{Y_{mn}}\big(-\pi_{mn}^*((m-2)E\big)\right)\otimes_R
	\Gamma(Y_{mn},\mathcal O_{Y_{mn}}(-D_{mn}))
	$$
	into the ideal $J=\Gamma(Y_{mn},\mathcal O_{Y_{mn}}(-\pi_{mn}^*((n-2)E)-D_{mn}))$ of $R$ is surjective.
	Thus 
		$J=\m_R^{m-2}I_{mn}$. 
	Moreover, for $m\gs 3$
	\begin{align*}
		I_n^m= \Gamma(Y_n,\mathcal O_{Y_n}(-D_n))^m
		&=\Gamma\left(Y_{mn},\mathcal O_{Y_{mn}}\big(-\pi_{mn,m}^*(D_n)\big)\right)^m\\
		&= \Gamma\left(Y_m,\mathcal O_{Y_{mn}}(-\pi_{mn}^*((2n+1)E)-E_1-\cdots-E_n)\right)^m\\
		&\subset\Gamma\left(Y_{mn},\mathcal O_{Y_{mn}}\big(-\pi_{mn}^*(m(2n+1)E)-mE_1\cdots-mE_n)\big)\right)\\
		&\subset\Gamma\left(Y_{mn},\mathcal O_{Y_{mn}}(-\pi_{mn}^*((m(2n+1)-1)E)-E_1-E_2-\cdots-E_{mn})\right)\\
		&=\Gamma\left(Y_{mn},\mathcal O_{Y_{mn}}(-\pi_{mn}^*((m-2)E)-D_{mn})\right)\\
        &=J=\m_R^{m-2}I_{mn}.
	\end{align*}
    Therefore
	$
	\sqrt{\m_RR[\II]}=\m_R\oplus R[\II]_{>0}
	$, and so
	 $\ell(\II)=0$, finishing the proof of part (2). 

	Finally, we prove part (3). 
  The $Y_n$ satisfy the assumptions of  \autoref{Theorem4} since $I_n\mathcal O_{Y_n}$ is invertible and $Y_n$ is normal for all $n$. 
     For every $v\in \Div(\m_R)\setminus \{v_0,v_1,\ldots,v_n\}$  the center $C(Y_n,v)$ is a point on $Y_n$ (of dimension 0). Furthermore 
	$$
	((I_n\mathcal O_{Y_n})\cdot C(Y_n,v_i))=\left(-D_n\cdot E_i(n)\right)=
    \begin{cases}
n+1&\text{ if }i=0,\\
1& \mbox{ if }1\ls i\ls n.
\end{cases}
   $$
In conclusion, for $v\in \Div(\m_R)$ we have
	\begin{equation}\label{eqNonlim2}
	((I_n\mathcal O_{Y_n}\cdot C(Y_n,v))=
    \begin{cases}
    n+1&\mbox{ if }v=v_0\\
		1&\mbox{ if }v=v_i\mbox{ and } 1\ls i\ls n\\
		0&\mbox{ otherwise.}
    \end{cases}
	\end{equation}
	Thus
	\begin{equation}\label{Ex5}
		\lim_{n\rightarrow\infty}\frac{((I_n\mathcal O_{Y_n})\cdot C(Y_n,v))}{n}=
        \begin{cases}
        1&\mbox{ if }v=v_0\\
			0&\mbox{ otherwise.}
        \end{cases}
	\end{equation}
	In particular, this limit always exists. From  \autoref{eqNonlim2} it follows that 
$$
\sum_{v\in {\rm Div}(\m_R)}v(x)\frac{((I_n\mathcal O_{Y_n})^{d-1}\cdot C(Y_n,v))}{n}=
v_0(x)+\frac{v_0(x)+v_1(x)+\cdots+v_n(x)}{n}.
$$
  Therefore, by \autoref{Theorem4}
	\begin{equation}\label{Ex3}
		e(\II(R/xR))
        =\lim_{n\rightarrow\infty}\left(v_0(x)+\frac{v_0(x)+v_1(x)+\cdots+v_n(x)}{n}\right).
	\end{equation}
 Let $C=\Div(x)$ be the divisor of $x$ on $\Spec(R)$, which  is a reduced curve on $\Spec(R)$.
Let  $\overline C$ be the strict transform of $C$ on $X$. Then $\overline{C}$ intersects $E$ in at most finitely many points, so there are at most finitely many of the points  $q_{i_1},\ldots,q_{i_r}$, indexed as  $i_1<\cdots<i_r$ in $\overline{C}$.  
	 Let $C'=C'(n)$ be the strict transform of $C$ on $Y_n$. 
 We have that $\pi^*(C)=\overline C+v_0(x)E$ and 
$$
(\pi\circ\pi_n)^*(C)=C'(n)+v_0(x)E_0(n)+\cdots+v_n(x)E_n(n)
$$
for all $n\gs 1$. Let $g=0$ be a local equation of $\overline C$ at $q_n$. Let $z,y$ be regular parameters in the regular local ring $\mathcal O_{X,q_n}$ such that $z=0$ is a local equation of $E$ at $q_n$. Then $x=gz^{v_0(x)}u$ where $u$ is a unit in $\mathcal O_{X,q_n}$. The valuation $v_n$ is the $\m$-adic valuation of $\mathcal O_{X,q_n}$,  then
$v_n(x)=v_n(g)+v_0(x)$. 
Since $q_n\in \overline C$ if and only if $n\in \{i_1,\ldots,i_r\}$, it follows that
$v_n(x)=v_0(x)$ if $n\in \{1,\ldots,n\}\setminus \{i_1,\ldots i_r\}$. 
Thus, by \autoref{Ex3}
$$e(\II(R/xR))=\lim_{n\rightarrow\infty}\left(v_0(x)+\frac{\sum_{j=1}^rv_{i_j}(x)+(n-r)v_0(x)}{n}\right)=2v_0(x).
	$$
	However, using \autoref{Ex5} we obtain 
	$$
	\sum_{v\in \Div(n_R)}v(x)\lim_{n\rightarrow\infty}\left(\frac{((I_n\mathcal O_{Y_n})\cdot C(Y_n,v))}{n}\right)=v_0(x).
    $$
\end{proof}

 \autoref{Theorem4} and \autoref{CorSwitch} raise  \autoref{quest:limit_exists} stated in the introduction. That is, 
under the notations in \autoref{Theorem4}, when does the limit
 $$
\lim_{n\rightarrow\infty}\frac{((I_n\mathcal O_{Y_n})^{d-1}\cdot C(Y_n,v))}{n^{d-1}}
$$
exists for $v\in \Div(\m_R)$? 
 Our next theorem shows that if $R$ 
 is a normal and excellent 
 $2$-dimensional local ring and $\II$ is a $\QQ$-divisorial filtration, then the limit in \autoref{quest:limit_exists} not only exists, but is given by a simple formula. For the statements and proof of the theorem we need some prior notation.

Let assumptions be as in the previous paragraph. 
There exists a birational projective morphism $\pi:X\rightarrow \Spec(R)$ such that $X$ is nonsingular and $\fm_R\mathcal O_X=\mathcal O_X(-D)$ where $D$ is a simple normal crossings 
divisor on $X$, so that all integral components of  $\pi^{-1}(\m_R)$ are nonsingular projective curves over $R/\m_R$. Moreover, there exists an effective  $\QQ$-divisor $\Delta$ on $X$ with exceptional support,  such that $\Gamma(X,\mathcal O_X(-\lceil n\Delta\rceil))=I_n$ for $n\gs 0$ and $-\Delta$ is a nef divisor on $X$. The existence of $X$ and $\Delta$ with these properties is shown in \cite{CNag}.

\begin{theorem}\label{TheoremLim} Let $(R,\fm_R,\kk)$ 
 be a normal and excellent 
 $2$-dimensional local ring and $\II=\{I_n\}_{n\gs 0}$ be a $\QQ$-divisorial filtration on $R$. 
Following the notation in the previous paragraph, let  $Y_n\rightarrow \Spec(R)$ be  projective birational morphisms such that $Y_n$ is normal and $I_n\mathcal O_{Y_n}$ is invertible. Then  
for every  $v\in \Div(\m_R)$ we have 
	$$
	\lim_{n\rightarrow\infty}\frac{(I_n\mathcal O_{Y_n}\cdot C(Y_n,v))}{n}=(-\Delta\cdot C(X,v)),
	$$
where $(\Delta\cdot C(X,v))$ is defined to be $\frac{1}{r}(r\Delta\cdot C(X,v))$ with $r\in \ZZ_{>0}$ and  $r\Delta$ an integral divisor.     
	In particular, this limit exists.
\end{theorem}
\begin{proof}
The result follows from  \autoref{PropLim1},  \autoref{CorLim},   and  \autoref{PropLim2} below.
\end{proof}

The following corollary is immediate from \autoref{TheoremLim} and  \autoref{CorSwitch}.

\begin{corollary}\label{dim2lim} Under the notations in \autoref{TheoremLim}, further assume that the union of the sets of Rees valuations of all the ideals $I_n$ is finite. For every $0\neq x\in  R$ with $R/xR$ is reduced we have
	$$
	e(\II(R/xR))=\sum_{v\in\Div(\m_R)}v(x)(-\Delta\cdot C(X,v)).
	$$
\end{corollary}

\begin{remark}\label{rem:dim2_assump}
The coefficients $(-\Delta\cdot C(X,v))$ are the analogues to the coefficients of $v(x)$  in \cite[\S 3]{ReSh}. We also note that  
there exist non-Noetherian divisorial filtrations of $\m_R$-primary ideals in  $2$-dimensional normal local rings satisfying the assumptions of  \autoref{dim2lim} as shown in  \autoref{ExS1}.  
There also exist (necessarily non-Noetherian) divisorial filtrations of $\m_R$-primary ideals $\II=\{I_n\}_{n\gs 0}$ in  $2$-dimensional normal local rings such that 
the union of the sets of Rees valuations of all the ideals $I_n$  is infinite, as shown in  \autoref{ExS2}. 
Finally, there exist filtrations  of $\m_R$-primary ideals in  $2$-dimensional regular local rings such that the limit and sum of  \autoref{Theorem4} do not compute as show in  \autoref{ExS0}.
\end{remark}

We continue with the notation introduced in the paragraph right before \autoref{TheoremLim}. 
Let $D\gs 0$ be an effective $\ZZ$-divisor on $X$. The {\it base locus} 
$$
\Bs(\Gamma(X,\mathcal O_X(-D))=\{p\in X\mid \Gamma(X,\mathcal O_X(-D))\mathcal O_{X,p}\ne \mathcal O_X(-D)_p\}
$$
is a closed subset of $X$ contained in $\pi^{-1}(\m_R)$.
We can write $-D=-M+F$ where $M$, $F$ are effective $\ZZ$-divisors such that
$\Gamma(X,\mathcal O_X(-D))=\Gamma(X,\mathcal O_X(-M))$ and the base locus $\Bs(\Gamma(X,\mathcal O_X(-M))$ is a finite set of closed points.   The divisor  $F$ is  the {\it fixed component} of $-D$. 
The {\it stable base locus} of $-D$ is defined as
\begin{equation}\label{eqN21}
\sBL(-D)=\bigcap_{n>0}\Bs(\Gamma(X,\mathcal O_X(-nD))).
\end{equation}
The stable base locus of a divisor on a projective variety is defined and its basic properties developed in \cite[Ch. 2]{Laz1}.    

The next proposition is the first step of the proof of \autoref{TheoremLim}.

\begin{proposition}\label{PropLim1}
Under the notations in \autoref{TheoremLim}, 
for every  $v\in \Div(\m_R)$ we have 
	\begin{equation*}\label{eqLim1}
		\limsup_{n\rightarrow\infty}\frac{(I_n\mathcal O_{Y_n}\cdot C(Y_n,v))}{n}\ls (-\Delta\cdot C(X,v)).
	\end{equation*}
\end{proposition}

\begin{proof} 
Fix $v\in \Div(\m_R)$. Let $r\in \ZZ_{>0}$ be such that  $r\Delta$ is an integral divisor. 
	We claim that there exists a birational projective morphism $a:X'\rightarrow X$ such that $X'$ is nonsingular and $C(X',v)$ is a projective curve 
    over $\kk=R/\m_R$. To prove this claim, we 
	 suppose that $C(X,v)$ is a closed point on $X$. 
     Then the center of $v$ on $R$ is $\m_R$. 
     Let $\mathcal O_v$ be the valuation ring of $v$, with maximal ideal $\m_v$. Thus $\m_v\cap R=\m_R$, giving an inclusion $R/\m_R\rightarrow \mathcal O_v/\m_v$. 
     Since $\dim(R)=2$ and $v$ is divisorial, $\trdeg_{R/\m_R}\left(\mathcal O_v/\m_v\right)=1$, so there exists $t\in \mathcal O_v$ such that the class of $t$ in $\mathcal O_v/\m_v$ is transcendental over $R/\m_R$. 
     Since $t$ is in the quotient field of $R$, there exist $f,g\in R$ such that $t=\frac{f}{g}$. 
     Let $X'$ be a resolution of singularities of the blowup of the ideal sheaf $(f,g)\mathcal O_X$. There exists an affine open subset $T=\Spec(S)$ of $X'$ such that 
	$
	t=\frac{f}{g}\in S
	$ 
     and $S\subset \mathcal O_v$.
	Let $Q=\m_v\cap S$. From the inclusions
    $$
   \kk \rightarrow S/Q\rightarrow \mathcal O_v/\m_v
    $$
    we see that $\trdeg_{\kk}\left(\Quot(S/Q)\right)=1$, 
where $\Quot(S/Q)$ is the quotient field of $S/Q$.
    By applying the dimension formula in \cite[Theorem 15.6]{Mat} to the inclusion $R\hookrightarrow S$ and $\m_R=Q\cap R$, we obtain  $$\height(Q)+\trdeg_{\kk}\left(\Quot(S/Q\right)=\dim(R)=2,$$
    and then $\dim(S/Q)=\dim(S)-\height(Q)=2-1=1$.  
    Since 	$S/Q$ is the coordinate ring of $C(X',v)\cap T$, it follows that $C(X',v)$ is a curve, completing the proof of the claim.
	
	Let $\Delta'=a^*(\Delta)$. Then $\Delta'$ is an effective $\QQ$-divisor on $X'$ such that 
	$$
	\Gamma(X',\mathcal O_{X'}(-\lceil n\Delta'\rceil))=\Gamma(X,\mathcal O_X(-\lceil n\Delta\rceil))=I_n
	$$
	for all $n\gs 0$ and $-\Delta'$ is nef on $X'$. 
    
    For $n\in \ZZ_{>0}$, write $n=m(n)r+s(n)$ with $m(n), s(n)$ nonnegative integers, and $0\ls s(n)< r$. 
    We thus have that $-\lceil n\Delta'\rceil =-m(n)r\Delta'-\lceil s(n)\Delta'\rceil$.
	Write $-\lceil n\Delta'\rceil =-M_n+G_n$ where $G_n$ is the fixed component of $-\lceil n\Delta'\rceil$.
	The fixed component of $-m(n)r\Delta'$ is bounded  by an effective exceptional divisor $A$ for $n> 0$ by  \cite[Proposition 5.3]{CNag}, and the fixed component of $-\lceil s(n)\Delta'\rceil$ 
    is bounded by an effective exceptional divisor $B$ for every $n$.  Thus the fixed component $G_n$ of $-\lceil n\Delta'\rceil$ is bounded by $A+B$ for all $n\gs 1$. 
	
	For $n\in \ZZ_{>0}$, let $f_n:Z_n\rightarrow X'$ be 
    a  birational projective $R$-morphism such that $Z_n$ is nonsingular, there is an $R$-morphism $h_n:Z_n\rightarrow Y_n$, and
    $I_n\mathcal O_{Z_n}$ is an invertible sheaf.
    We have that
    $$
    \Gamma(X',\mathcal O_{X'}(-M_n))=\Gamma(X',\mathcal O_{X'}(-M_n+G_n))=I_n,
    $$
    and since $\Gamma(X',\mathcal O_{X'}(-M_n))$ has  no fixed component, $I_n\mathcal O_{X'}=\mathcal J_n\mathcal O_{X'}(-M_n)$, where $\mathcal J_n$ is an ideal sheaf of $X'$ such that $\dim \Supp(\mathcal O_{X'}/\mathcal J_n)\ls 0$. Then since $I_n\mathcal O_{Z_n}$ is invertible, 
    $$
    I_n\mathcal O_{Z_n}=\mathcal J_nf_n^*\mathcal O_{X'}(-M_n)
    =f_n^*\mathcal O_{X'}(-M_n)\mathcal O_{Z_n}(-F_n)=\mathcal O_{Z_n}(-f_n^*M_n-F_n)
    $$
    where $F_n$ is the effective divisor on $Z_n$ such that $\mathcal J_n\mathcal O_{Z_n}=\mathcal O_{Z_n}(-F_n)$. Thus $F_n$ is effective and exceptional for $f_n$ since $\dim \Supp(\mathcal O_{X'}/\mathcal J_n)\ls 0$, so that $f_n(F_n)$ is a finite union of closed points on $X'$. 
      In particular, $C_n:=C(Z_n,v)$ is not a component of $F_n$  since  $f_n(C_n)=C(X',v)$ is a curve.  Thus $(F_n\cdot C_n)=\deg(\mathcal O_{Z_n}(F_n)\otimes\mathcal O_{C_n})\gs 0$.	Therefore
	$$
	(I_n\mathcal O_{Z_n}\cdot C_n)=(-f_n^*(M_n)\cdot C_n)-(F_n\cdot C_n)
	\ls (-M_n\cdot (f_n)_*(C_n))=((-\lceil n\Delta'\rceil-G_n)\cdot C(X',v)).
	$$
	The effective divisors $G_n$ are bounded for $n>0$, so that $\{G_n\}_{n>0}$ is a finite set.
	Thus there exist $b,b'\in \ZZ_{>0}$ such that 
			$$(I_n\mathcal O_{Z_n}\cdot C_n)\ls (-m(n)r\Delta'\cdot C(X',v))+b'
			=(-m(n)r\Delta\cdot C(X,v))+b'\ls n(-\Delta \cdot C(X,v))+b,$$
    where the equality follows from the projection formula of \cite[Proposition 2.6(c)]{CM}.
	
   We have that $h_n^*(I_n\mathcal O_{Y_n})\cong I_n\mathcal O_{Z_n}$.
	Thus
	$$
		(I_n\mathcal O_{Y_n}\cdot C(Y_n,v))=
		(I_n\mathcal O_{Y_n}\cdot (h_n)_*(C_n))
		=((h_n)^*(I_n\mathcal O_{Y_n})\cdot C_n)=
		(I_n\mathcal O_{Z_n}\cdot C_n)
	$$
	by the projection formula of \cite[Proposition 2.6(c)]{CM}.
	Thus for all $n\in \ZZ_{n>0}$,
	$$
	(I_n\mathcal O_{Y_n}\cdot C(Y_n,v))\ls n(-\Delta\cdot C(X,v))+b,
	$$
	whence the result follows.
\end{proof}

We obtain the following immediate corollary.

\begin{corollary}\label{CorLim}
Under the notations in \autoref{PropLim1}, if $(-\Delta\cdot C(X,v))=0$, then 
	$$
	\lim_{n\rightarrow \infty}\frac{(I_n\mathcal O_{Y_n}\cdot C(Y_n,v))}{n}=0.
	$$
\end{corollary}

We continue with the second step of the proof of the theorem.

\begin{proposition}\label{PropLim2}
Under the notations in \autoref{TheoremLim}, 
if $(-\Delta\cdot C(X,v))>0$, then
	$$
	\liminf_{n\rightarrow\infty} \frac{(I_n\mathcal O_{Y_n}\cdot C(Y_n,v))}{n}\gs (-\Delta \cdot C(X,v)).
	$$
\end{proposition}

\begin{proof} 
Fix $v\in \Div(\m_R)$. We borrow the notation used in the proof of \autoref{PropLim1}. 
 By our construction of $X'$, $C'=C(X',v)$ is a  projective curve over the field
$\kk=R/\m_R$. Moreover, $\fm_R\mathcal O_{X'}=\mathcal O_{X'}(-D')$, where $D'$ is a SNC divisor on $X'$ since $\fm_R\mathcal O_X=\mathcal O_X(-D)$ and $D$ is a SNC divisor on $X$, $X'\rightarrow X$ is birational, and $X', X$ are nonsingular  of dimension 2.

Relabel $X'$ as $X$,  $\Delta'$ as $\Delta$, and $C'$ as  $C$, for simplicity of notation. 
We also use the notation and results of \autoref{SecRR} on the Riemann-Roch Theorem for projective curves  
regarding $C$ as a projective curve over $\kk$. 

From the
short exact sequences
	$$
	0\rightarrow \mathcal O_X(-\lceil n\Delta \rceil -C)\rightarrow \mathcal O_X(-\lceil n\Delta\rceil)\rightarrow \mathcal O_X(-\lceil n\Delta\rceil)\otimes\mathcal O_C\rightarrow 0
	$$
    we obtain the exact sequences
	\begin{align}\label{eqN20}
	\Gamma(X,\mathcal O_X(-\lceil n\Delta\rceil))\rightarrow \Gamma(C,\mathcal O_X(-\lceil n\Delta\rceil)\otimes\mathcal O_C)
    \rightarrow H^1(X,\mathcal O_X(-\lceil n\Delta\rceil-C)).
    \end{align}
	There exists a constant $e>0$ such that 
	\begin{equation}\label{eqLim4}
		\ell_R\left(H^1(X,\mathcal O_X(-\lceil n\Delta\rceil-C)\right)<e
	\end{equation}
	for $n>0$ by  \cite[Corollary 5.7]{CNag}. Moreover $h^0\left(\mathcal O_X(-\lceil n\Delta\rceil)\otimes\mathcal O_C\right)$ goes to infinity  as $n\rightarrow\infty$ by the Riemann-Roch Theorem,  recalled in    \mbox{\autoref{SecRR}}, 
    and since
	$$
	\deg\left(\mathcal O_X(-\lceil n\Delta\rceil)\otimes\mathcal O_C\right)=(-\lceil n\Delta\rceil\cdot C) \quad \text{and}\quad (-\Delta\cdot C)>0.
	$$
	  Therefore $\Gamma(X,\mathcal O_X(-\lceil n\Delta\rceil-C))\ne \Gamma(X,\mathcal O_X(-\lceil n\Delta\rceil))$ for $n$ sufficiently large, and so 
    there exists $n_0$ such that $C$ is not a component of the fixed component $G_n$ of $-\lceil n\Delta\rceil$  for $n\gs n_0$. We now restrict to $n\gs n_0$.
    
  Let $\sigma_n\in \Gamma(X,\mathcal O_X(G_n))$ be the section such that $\Div(\sigma_n)=G_n$ and let 
  $$
  \mathcal O_X(-\lceil n\Delta\rceil-G_n)
  \xrightarrow{\sigma_n}
  \mathcal O_X(-\lceil n\Delta\rceil)
  $$
  be the natural injection. Then
  $$
  0\rightarrow \mathcal O_X(-\lceil n\Delta\rceil-G_n)\otimes\mathcal O_C
  \xrightarrow{\sigma_n\otimes 1}
  \mathcal O_X(-\lceil n\Delta\rceil)\otimes\mathcal O_C
  $$
 is an injection by  \autoref{LemInj} since $C$ is not in the support of $\Div(\sigma_n)$ as $n\ge n_0$ by assumption, and so we have a commutative diagram
\begin{equation}\label{D1}
\begin{tikzcd}
&\Gamma(X,\mathcal O_X(-\lceil n\Delta\rceil-G_n)) \arrow[r, "\sim"] \arrow[d] 
  & \Gamma(X,\mathcal O_X(-\lceil n\Delta\rceil)) \arrow[d] \\
0\arrow[r] &\Gamma(C,\mathcal O_X(-\lceil n\Delta\rceil-G_n)\otimes\mathcal O_C) \arrow[r] 
  & \Gamma(X,\mathcal O_X(-\lceil n\Delta\rceil)\otimes\mathcal O_C).
\end{tikzcd}
  \end{equation} 
Thus
	$$
	U_n:=\Image\big(\Gamma(X,\mathcal O_X(-\lceil n\Delta\rceil))\rightarrow \Gamma(C,\mathcal O_X(-\lceil n\Delta\rceil)\otimes\mathcal O_C)\big)
    \subset \Gamma(C,\mathcal O_X(-\lceil n\Delta\rceil-G_n)\otimes\mathcal O_C)). 
	$$

	There exists a sequence of morphisms $\pi_{i+1}:X_{i+1}\rightarrow X_i$ for $i\gs 0$ such that $\pi_{i+1}$ is the blowup of a closed point $p_{i+1}$ on the strict transform $C_{i}$ of $C$ on $X_i$, with exceptional divisor $E_{i+1}$, and such that for $n\gs n_0$, there exists a positive integer $\sigma(n)$ with $I_n\mathcal O_{X_{\sigma(n)}}$ invertible in a neighborhood of $C_{\sigma(n)}$. This is constructed as follows. Let $Z_1$ be a resolution of singularities of the blowup of $I_1\mathcal O_X$. Since $X$ is nonsingular, $Z_1\rightarrow X$ is a sequence of blowups of closed points by \cite[Theorem 4.1]{L}. 
    Let $T_1\rightarrow X$ be the sequence of blowups of points which is obtained by only blowing up points in the sequence $Z_1\rightarrow X$ which are on the strict transform of $C$. Then $I_1\mathcal O_{T_1}$ is invertible in a neighborhood of the strict transform of $C$ on $T_1$. Let $Z_2$ be a resolution of singularities of the blowup of $I_2\mathcal O_{T_1}$. Then $Z_2\rightarrow T_1$ is again a sequence of blowups of closed points, and letting $T_2\rightarrow T_1$ be  the sequence of blowups of points which is obtained by only blowing up points in the sequence $Z_2\rightarrow T_1$ which are on the strict transform of $C$, we obtain that $I_2\mathcal O_{T_2}$ is invertible. By  induction, we  construct the morphisms $\pi_{i+1}:X_{i+1}\rightarrow X_i$.


Assume that 
$$
	p_1\in \Bs(\Gamma(X,\mathcal O_X(-\lceil n\Delta\rceil-G_n))\cap C.
	$$ 
    Then the fixed component of $\pi_1^*(-\lceil n\Delta\rceil-G_n)$ is $a_1(n)E_1$ for some $a_1(n)>0$, and so the fixed component of $\pi_1^*(-\lceil n\Delta\rceil)$ is $\pi_1^*(G_n)+a_1(n)E_1$. Let $C_1\cong C$ be the strict transform of $C$ on $X_1$ (the Zariski closure of $\pi_1^{-1}(C\setminus \{p_1\})\cong C\setminus \{p_1\}$ in $X_1$).  Since $C_1$ is not in the support of the divisor $a_1(n)E_1$, using  \autoref{LemInj} again we obtain a commutative diagram with $\mathcal L_n:=\mathcal O_X(-\lceil n\Delta\rceil-G_n)$ and $\mathcal M_n:=\mathcal O_{X_1}(-a_1(n)E_1)$ 
    \begin{equation}\label{D2}
    \begin{tikzcd}
    &\Gamma(X_1,\pi_1^*\mathcal L_n\otimes \mathcal M_n) \arrow[r, "\sim"] \arrow[d] 
    & \Gamma(X_1,\pi_1^*\mathcal L_n) \arrow[r, "\sim"] \arrow[d] & \Gamma(X,\mathcal L_n)\\
    0\arrow[r] &\Gamma(C_1,\pi_1^*\mathcal L_n\otimes \mathcal M_n\otimes\mathcal O_{C_1}) \arrow[r] \arrow[d] 
    & \Gamma(C_1,\pi_1^*\mathcal L_n\otimes\mathcal O_{C_1}) \arrow[d] &\\
    0\arrow[r] &\Gamma(C,\mathcal L_n\otimes\mathcal O_C(-a_1(n)p_1)) \arrow[r] 
    & \Gamma(C,\mathcal L_n\otimes\mathcal O_C),&
    \end{tikzcd}
     \end{equation}
	where we are using the isomorphism 
	$(\pi_1)_*(\mathcal M_n \otimes\mathcal O_{C_1})\cong \mathcal O_C(-a_1(n)p_1)$. 
	Combining diagrams \autoref{D1} and \autoref{D2} we have a natural  inclusion
\begin{equation}\label{eq:Un_contain}
	U_n\subset \Gamma(C,\mathcal O_X(-\lceil n\Delta\rceil-G_n)\otimes\mathcal O_C(-a_1(n)p_1)).
	\end{equation} 
	If $p_1\not\in \Bs(\Gamma(X,\mathcal O_X(-\lceil n\Delta\rceil-G_n))$, then \autoref{eq:Un_contain} holds with $a_1(n)=0$.

	By induction on $n$ and the argument leading to the inclusion  \autoref{eq:Un_contain}, we obtain  	
		\begin{align}\label{eqLim3}			U_n&\subset \Gamma(C,\mathcal O_X(-\lceil n\Delta\rceil)\otimes\mathcal B_n^{-1}\otimes \mathcal O_C(-a_1(n)p_1-\cdots - a_{\sigma(n)}(n)p_{\sigma(n)}))
		\end{align}
where $\mathcal B_n=\mathcal O_X(G_n)\otimes\mathcal O_C$.
 There exists a positive integer $r$ such that $r\Delta$ is an integral divisor and since $(-\Delta\cdot C)>0$,  we can also replace  $r$ with a multiple of $r$ so that 
$
\deg(\mathcal O_X(-r\Delta)\otimes\mathcal O_C)>2p_a(C)-2
$. 
Let 
$$
f=e+\deg(\mathcal O_X(-r\Delta)\otimes\mathcal O_C)
$$
where $e$ is defined by \autoref{eqLim4}.

  Let $\mathcal E_n:=\mathcal B_n\otimes\mathcal O_C(a_1(n)p_1+\cdots+a_{\sigma(n)}(n)p_{\sigma(n)})$,	  we  now show that  
	\begin{align}\label{eqLim7}
	0&\ls \deg\left(\mathcal E_n\right)
	=\deg(\mathcal B_n)+a_1(n)[R(p_1):\kk]+\cdots+a_{\sigma(n)}(n)[R(p_{\sigma(n)}):\kk]\ls  f,
     \end{align}
	for all $n\gg0$,  where $R(p_i)$ is the residue field of 
	$\mathcal O_{X,p_i}$ 
    for each $i$. 
	Note that  
	\begin{align*}
	\dim_{\kk}(\Gamma(C,\mathcal O_X(-\lceil n\Delta\rceil)\otimes\mathcal O_C)/U_n)	&=\ell_R(\Gamma(C,\mathcal O_X(-\lceil n\Delta\rceil)\otimes\mathcal O_C)/U_n)\\
	&\ls
	\ell_R(H^1(X,\mathcal O_X(-\lceil n\Delta\rceil-C))< e
	\end{align*}
     by (\ref{eqN20}). 
	Thus, using \autoref{eqLim3}, we have 
		\begin{align}\label{eqLim6}
			e> &h^0(\mathcal O_X(-\lceil n\Delta\rceil)\otimes\mathcal O_C)-h^0(\mathcal O_X(-\lceil n\Delta\rceil)\otimes\mathcal E_n^{-1}).
		\end{align}
%
 Suppose that \autoref{eqLim7} does not hold, i.e., $\deg(\mathcal E_n)>f$.  We will derive a contradiction.  Since $(-\Delta\cdot C)>0$, there exists $n_1$ such that
	$
	\deg(\mathcal O_X(-\lceil n\Delta\rceil)\otimes\mathcal O_C)>2p_a(C)-2
	$
	for $n\gs n_1$, so that 
	$$
	h^0(\mathcal O_X(-\lceil n\Delta\rceil)\otimes\mathcal O_C)=\chi(\mathcal O_X(-\lceil n\Delta\rceil)\otimes\mathcal O_C))
	$$
	for  $n\gs n_1$ by the Riemann-Roch Theorem. 
	First  suppose  that $n\gs n_1$ 
    and $$\deg(\mathcal O_X(-\lceil n\Delta\rceil)\otimes\mathcal E_n^{-1})>2p_a(C)-2.$$
Then, from \autoref{eqLim6}
\begin{align*}
e&> h^0(\mathcal O_X(-\lceil n\Delta\rceil)\otimes\mathcal O_C)
-h^0(\mathcal O_X(-\lceil n\Delta\rceil)\otimes\mathcal E_n^{-1})\\
&= \chi(\mathcal O_X(-\lceil n\Delta\rceil)\otimes\mathcal O_C)-\chi(\mathcal O_X(-\lceil n\Delta\rceil)\otimes\mathcal E_n^{-1})\,\,=\,\,\deg(\mathcal E_n)>e,
\end{align*}
a contradiction.	 Now consider   $n\gs n_1$ and the remaining case, 
	$$
	\deg(\mathcal O_X(-\lceil n\Delta\rceil)\otimes\mathcal E_n^{-1})\ls 2p_a(C)-2.
	$$ 
Then
$$
h^0(\mathcal O_X(-\lceil n\Delta\rceil)\otimes\mathcal E_n^{-1})\ls h^0(\mathcal O_X(-\lceil n\Delta\rceil-r\Delta)\otimes\mathcal E_n^{-1})=\chi(\mathcal O_X(-\lceil n\Delta\rceil-r\Delta)\otimes\mathcal E_n^{-1}).
$$		
Therefore		
		\begin{align*}
		e&>h^0(\mathcal O_X(-\lceil n\Delta\rceil)\otimes\mathcal O_C)
-h^0(\mathcal O_X(-\lceil n\Delta\rceil)\otimes\mathcal E_n^{-1})\\
&= \chi(\mathcal O_X(-\lceil n\Delta\rceil)\otimes\mathcal O_C)-h^0(\mathcal O_X(-\lceil n\Delta\rceil)\otimes\mathcal E_n^{-1})\\
&\gs \chi(\mathcal O_X(-\lceil n\Delta\rceil)\otimes\mathcal O_C)-\chi(\mathcal O_X(-\lceil n\Delta\rceil-r\Delta)\otimes\mathcal E_n^{-1})\\
&= -\deg(\mathcal O_C(-r\Delta))+\deg(\mathcal E_n)>-\deg( \mathcal O_C(-r\Delta))+f=e,
	\end{align*}
 a contradiction. 
 Thus \autoref{eqLim7} holds for  $n\gs n_1$.

	Let $\tau_{n,m}:X_{\sigma(n)}\rightarrow X_m$ be the natural morphisms for $\sigma(n)\gs m$. 
	Since $I_n\mathcal O_{X_{\sigma(n)}}$ is invertible in a neighborhood of $C_{\sigma(n)}$, we have 
    \begin{align*}
    I_n\mathcal O_{X_{\sigma(n)}}
	=
    \mathcal O_{X_{\sigma(n)}}\big(-\tau_{n,0}^*(-\lceil n\Delta\rceil-G_n)-
    \sum_{i=1}^{\sigma(n)-1}a_i(n)\tau_{n,i}^*(E_i)
   -a_{\sigma(n)}(n)E_{\sigma(n)}\big)\mathcal I_n,
    \end{align*}
	where $\mathcal I_n$ is an ideal sheaf on $X_{\sigma(n)}$ such that the  support of $\mathcal O_{X_{\sigma(n)}}/\mathcal I_n$  is disjoint from $C_{\sigma(n)}$. Let $\lambda_n:V_n\rightarrow X_{\sigma(n)}$ be a resolution of singularities of the blowup of $\mathcal I_n$, so that  $I_n\mathcal O_{V_n}$ is invertible. 
	Let $\tilde{C}$ be the strict transform of $C_{\sigma(n)}$ on $V_n$. 
	Then
    \begin{align*}
    	&(I_n\mathcal O_{V_n}\cdot C(V_n,v))
			=(I_n\mathcal O_{V_n}\cdot \tilde{C})\\
			&=\deg(I_n\mathcal O_{V_n}\otimes\mathcal O_{\tilde{C}})=\deg(I_n\mathcal O_{X_{\sigma(n)}}\otimes\mathcal O_{C_{\sigma(n)}})\\
          &=\deg\big(\mathcal O_{X_{\sigma(n)}}\big(-\tau_{n,0}^*(-\lceil n\Delta\rceil-G_n)-
    \sum_{i=1}^{\sigma(n)-1}a_i(n)\tau_{n,i}^*(E_i)
   -a_{\sigma(n)}(n)E_{\sigma(n)}\big)\otimes\mathcal O_{C_{\sigma(n)}}\big)\\  
	&=\deg(\mathcal O_X(-\lceil n\Delta\rceil)\otimes \mathcal E^{-1}_n
   ))
	\gs (-\lceil n\Delta\rceil\cdot C(X,v))- f.
    \end{align*}
	Thus
	$$
	\liminf_{n\rightarrow\infty}\frac{(I_n\mathcal O_{Y_n}\cdot C(Y_n,v))}{n}
	= \liminf_{n\rightarrow\infty}\frac{(I_n\mathcal O_{V_n}\cdot C(V_n,v))}{n}
		\gs (-\Delta\cdot C(X,v)).
	$$
\end{proof}


\section{When is the set of Rees valuations of the members of a graded family finite?}

We now consider the question of when a graded family of ideals $\II=\{I_n\}_{n\gs 0}$ has only finitely many Rees valuations; that is, when the union of the sets of Rees valuations of all the $I_n$ is a finite set.
To begin with, we observe that if $\II$ is a Noetherian filtration, then under mild assumptions on the ring  the union of the sets of Rees valuations of all the ideals $I_n$ is finite.

\begin{proposition}\label{prop:Rees_finite} Let $(R,\fm_R,\kk)$ be  an analytically unramified local domain 
 and $\II=\{I_n\}_{n\gs 0}$ be a graded family of ideals such that $R[\II]$ is Noetherian. Then  the union of the sets of Rees valuations of all the ideals $I_n$ is finite. 
\end{proposition}

\begin{proof} The Rees algebra $R[\II]$ is a finitely generated $R$-algebra since it is Noetherian. There exists $d>0$ such that $R[\II]^{(d)}:=\oplus_{n\gs 0}\,I_{nd}$ is generated in degree 1 by \cite[Proposition III.1.3]{Bou}. By \cite[Lemma III.1.1]{Bou}, there exists $m_0>0$ such that $I_{nd+k}=I_{nd}I_k$ for $m_0d\ls k<m_0(d+1)$ and $n\gs 0$.  Since $R$ is analytically unramified, there exists by \cite[Corollary 9.2.1]{HS} an $R$-morphism 
$\pi:X\rightarrow\Spec(R)$ such that $X$ is normal, projective over $\Spec(R)$, and $I_l\mathcal O_X$ is invertible for $l<m_0(d+1)$. Then $I_n\mathcal O_X$ is invertible for all $n\gs 0$.

Let $\pi_n:X_n\rightarrow \Spec(R)$ be the normalized blowup of $I_n$, which is a projective morphism by \cite[Corollary 9.2.1]{HS}. The Rees valuations of $I_n$ are the equivalence classes of valuations $v_{n,i}$ for $1\ls i\ls \lambda_n$ whose valuation ring is $\mathcal O_{X_n,E_{n,i}}$, where 
$$
I_n\mathcal O_{X_n}=\mathcal O_{X_n}(-a_1E_{n,1}-\cdots-a_{n,\lambda_n}E_{n,\lambda_n})
$$
with $E_{n,i}$ being prime Weil divisors on $X_n$. Since $I_n\mathcal O_X$ is invertible, there exist natural projective $R$-morphisms $\phi_n:X\rightarrow X_n$ for $n\gs 1$. Now $I_1^n\subset I_n$ for all $n$ so that we have an inclusion of supports $\Supp(R/I_1)\supset \Supp(R/I_n)$ for all $n$. The strict transforms of the $E_{n,i}$ are prime divisors on $X$ that are contained in $\pi^{-1}(\Supp(R/I_1))$, which contains only a finite number of prime divisors of $X$. Thus the union of all the Rees valuations of all the $I_n$ is a finite set.
\end{proof}

In the remainder of this section we show the examples mentioned in \autoref{rem:dim2_assump}. In \autoref{ExS1} we prove the existence of non-Noetherian divisorial filtrations of $\m_R$-primary ideals in  $2$-dimensional normal local rings satisfying the condition of  \autoref{dim2lim}, i.e.,   with finite total number of Rees valuations. On the other hand, in \autoref{ExS2}  we prove that not all   divisorial filtrations of $\m_R$-primary ideals in  $2$-dimensional normal local rings satisfy this condition. 

The graded filtration $I_n=\m_R$ for $n>0$ gives a simple example of a non-Noetherian graded family such that the union of the sets of Rees valuations of all the ideals $I_n$ is finite.

\begin{example}\label{ExS1} There exists a non-Noetherian divisorial filtration of $\m_R$-primary ideals $\II=\{I_n\}_{n\gs 0}$ in a $2$-dimensional  normal local ring $R$ that is essentially of finite type over a field such that the union of the sets of Rees valuations of all the ideals $I_n$ is finite. 
\end{example}

\begin{proof} Let $C$ be an elliptic curve over an algebraically closed field and $p\in C$ be a closed point. Let $Z=\PP(\mathcal O_C\oplus\mathcal O_C(-p))$ be the projective bundle over $C$
	with natural projection $\lambda:Z\rightarrow C$. Let $C_0$ be the zero section of $\lambda$, so that 
	$\mathcal O_{C_0}\otimes\mathcal O_Z(-C_0)\cong \mathcal O_C(-p)$. Let $Z\rightarrow W$ be the contraction of $C_0$ to a closed point $a$ on a normal surface $W$, which exists by  \cite[Proposition II.8.8.2 \& Remark II.8.8.3]{EGA}. Let $R=\mathcal O_{W,a}$ and $X=Z\times_W\Spec(R)$  with natural birational projective morphism $\pi:X\rightarrow \Spec(R)$. We note that 
	$R$ is a $2$-dimensional  normal local ring such that $X\rightarrow \Spec(R)$ is a resolution of singularities and $X\setminus C_0\rightarrow \Spec(R)\setminus \{\m_R\}$ is an isomorphism. 
	
	Let $q\in C_0$ be such that $\mathcal O_C(p-q)$ has infinite order in the Jacobian of $C$, and let $\tau:Y\rightarrow X$ be the blowup of $q$ with exceptional divisor $E$. Let $C^*$ be the strict transform of $C_0$ on $Y$. Therefore
	$$
    E\cong \PP^1, \qquad  \mathcal O_E\otimes\mathcal O_Y(E)\cong \mathcal O_{\PP^1}(-1), \qquad \text{and}\qquad
	\mathcal O_E\otimes\mathcal O_Y(C^*)\cong \mathcal O_{\PP^1}(1).
    $$
	We  have  $C^*\cong C$ and 
	$$
	\mathcal O_C(-p)\cong \mathcal O_{C_0}\otimes\mathcal O_X(C_0)\cong \mathcal O_{C^*}\otimes\mathcal O_Y(\tau^*C_0)\cong \mathcal O_{C^*}\otimes\mathcal O_Y(C^*+E),
	$$
	so $\mathcal O_{C^*}\otimes\mathcal O_Y(C^*)\cong \mathcal O_C(-p-q)$. 
	Let $\Delta=C^*+2E$, then
    $$\mathcal O_{C^*}\otimes\mathcal O_Y(-\Delta)\cong \mathcal O_C(p-q)\qquad\text{and}\qquad
	\mathcal O_E\otimes\mathcal O_Y(-\Delta)\cong \mathcal O_{\PP^1}(1),$$ so that $-\Delta$ is nef.

	Define the divisorial filtration of $\m_R$-primary ideals $\II=\{I_n\}_{n\gs 0}$ on $R$ by 
	$$
	I_n=\Gamma(Y,\mathcal O_{Y_n}(-n\Delta))
	$$
	for $n\gs 0$. 
	Note that if $\mathcal L$ is a line bundle on the elliptic curve $C$, then $H^1(C,\mathcal L)=0$ if $\deg(\mathcal L)\gs 1$ by the Riemann-Roch Theorem, since $C$ has genus 1. Moreover, if $\mathcal L'$ is a line bundle on $E$, then $H^1(E,\mathcal L')=0$ if $\deg(\mathcal L')\gs -1$, since $E\cong\PP^1$. 
	For integers $m\gs 0$ and $n\gs 2$, we have
	\begin{equation}\label{Neq1}
		H^1(E,\mathcal O_E\otimes\mathcal O_Y(-(m+1)C^*-mE-n\Delta-2C^*))=0,
	\end{equation}
	since $\deg(\mathcal O_E\otimes\mathcal O_Y(-(m+1)C^*-mE-n\Delta-2C^*))=n-3$, and  for $m\gs 0$,  $n\gs 0$, we have
	\begin{equation}\label{Neq2}
		H^1(C^*,\mathcal O_{C^*}\otimes\mathcal O_Y(-mC^*-mE-n\Delta-2C^*))=0
	\end{equation}
	since $\deg(\mathcal O_{C^*}\otimes\mathcal O_Y(-mC^*-mE-n\Delta-2C^*))=2(m+2)-m=m+4>0$.

	For every integer $m\gs 0$ we have short exact sequences 
    \begin{align}\label{Neq5}
    &0\rightarrow \O_E\otimes\mathcal O_Y(-(m+1)C^*-mE)\rightarrow \mathcal O_{(m+1)C^*+(m+1)E}\rightarrow \mathcal O_{(m+1)C^*+mE}\rightarrow 0\\
   \nonumber&0\rightarrow \mathcal O_{C^*}\otimes\mathcal O_Y(-mC^*-mE)\rightarrow \mathcal O_{(m+1)C^*+mE}\rightarrow \mathcal O_{mC^*+mE}\rightarrow 0. 
    \end{align} 
	Tensoring \autoref{Neq5} with $\mathcal O_Y(-n\Delta-2C^*)$ for  $n\gs 2$, taking cohomology, and using \autoref{Neq1} and \autoref{Neq2}, we obtain 
	\begin{equation*}\label{Neq3}
		H^1(Y,\mathcal O_{mC^*+mE}\otimes\mathcal O_Y(-n\Delta-2C^*))=0\qquad\text{for all $m\gs 0$ and $n\gs 2$.}
	\end{equation*} 
	 Thus for all $n\gs 2$,
	$$
	H^1(Y,\mathcal O_Y(-n\Delta-2C^*))\otimes_R\hat R\cong \lim_{\leftarrow }H^1(Y,\mathcal O_{mC^*+mE}\otimes\mathcal O_Y(-n\Delta-2C^*))=0
	$$
	by 
    \cite[Theorem III.11.1]{H}, where $\hat{R}$ is the $\fm_R$-adic completion of $R$ and the inverse limit is over $m\in \ZZ_{>0}$. Thus
    \begin{equation}\label{Neqq}
		H^1(Y,\mathcal O_Y(-n\Delta-2C^*))=0\qquad \text{for $n\gs 2$,} 
	\end{equation}
    since $\hat R$ is a faithfully flat $R$-module.
	
	Similarly, for $m\gs 0$ and $n\gs 2$, we have 
	$$
	H^1(E,\mathcal O_E\otimes\mathcal O_Y(-(m+1)C^*-mE-n\Delta-C^*-E))=0,
	$$
	since $E\cong\PP^1$ and 
	$\deg(\mathcal O_E\otimes\mathcal O_Y(-(m+1)C^*-mE-n\Delta-C^*-E))=n-1>-1$,  
	and for $m\gs 0$ and $n\gs 2$, we have
	$$
	H^1(C^*,\mathcal O_{C^*}\otimes \mathcal O_Y(-mC^*-mE-n\Delta-C^*-E))=0
	$$
	since
	$\deg(\mathcal O_{C^*}\otimes \mathcal O_Y(-mC^*-mE-n\Delta-C^*-E))=m+1>0$.
	Therefore,  tensoring \autoref{Neq5} with $\mathcal O_Y(-n\Delta-C^*-E)$ and proceeding as in the proof of \autoref{Neqq}, we obtain 
	\begin{equation}\label{Neq4}
		H^1(Y,\mathcal O_Y(-n\Delta-C^*-E))=0\qquad \text{for $n\gs 2$.} 
	\end{equation}

	Now $\Gamma(C^*,\mathcal O_Y(-n\Delta))=0$ for all $n\gs 0$ so that $C^*$ is contained in the stable base locus $\sBL(-\Delta)$. The natural map
	$$
	\Gamma(Y,\mathcal O_Y(-n\Delta-C^*))\rightarrow \Gamma(C^*,\mathcal O_{C^*}\otimes\mathcal O_Y(-n\Delta-C^*))
	$$
	is surjective for $n\gs 2$ by \autoref{Neqq}, and $\mathcal O_{C^*}\otimes\mathcal O_Y(-n\Delta-C^*)$ is generated by global sections for all $n\gs 0$ since $(C^*\cdot(-n\Delta-C^*))=-((C^*)^2)=2$ and $C$ is an elliptic curve.  Thus 
	$I_n\mathcal O_Y=\mathcal O_Y(-n\Delta-C^*)$ is invertible in a neighborhood of $C^*$ for $n\gs 2$.
	The natural map 
	$$
	\Gamma(Y,\mathcal O_Y(-n\Delta-C^*))\rightarrow \Gamma(E,\mathcal O_E\otimes\mathcal O_Y(-n\Delta-C^*))
	$$
	is surjective by \autoref{Neq4} and $\Gamma(E,\mathcal O_E\otimes\mathcal O_Y(-n\Delta-C^*))$ is generated by global sections for all $n\gs 0$ since $\deg(\mathcal O_Y(-n\Delta-C^*))=n-1$. Thus $I_n\mathcal O_Y$ is invertible in a neighborhood of $E$ for $n\gs 2$, and so $I_n\mathcal O_Y$ is invertible on a neighborhood of $\pi^{-1}(\m_R)$.  Thus $I_n\mathcal O_Y=\mathcal O_Y(-n\Delta-C^*)$ is invertible for $n\gs 2$.
	
	Since $((-n\Delta-C^*)\cdot C^*)=2>0$ and $((-n\Delta-C^*)\cdot E)=n-1>0$ for $n\gs 2$, 
    the natural morphism $Y\rightarrow \Proj(R[\II])$ does not contract $C^*$ or $E$ to points. Thus the natural valuations $v_E$ and $v_{C^*}$ associated to $E$ and $C^*$ respectively are the Rees valuations of $I_n$ for $n\gs 2$. Furthermore, 
    $-\Delta$ is a nef Cartier divisor and the stable base locus of $-\Delta$ is $\sBL(-\Delta)=C^*$, so $R[\II]$ is not a finitely generated $R$-algebra by  \cite[Lemma 8.3]{C4}. The stable base locus is defined in \autoref{eqN21}. 
\end{proof}

\begin{example}\label{ExS2} There exists a divisorial filtration of $\m_R$-primary ideals $\II=\{I_n\}_{n\gs 0}$ of a $2$-dimensional  normal local ring $A$ which is essentially of finite type over $\CC$, such that the set of all Rees valuations of all the $I_n$ is an infinite set. Necessarily, $\II$ is non-Noetherian.
\end{example}

\begin{proof}
	We construct the example using a method from  \cite[\S 2]{C5}. Let $C$ be an elliptic curve over $\CC$ and $a,q_1\in C$ be distinct closed points such that $\mathcal O_C(a-q_1)$ has infinite order in the Jacobian of $C$. Let $E\cong \PP^1_{\CC}$ and let $q_2,b,c\in E$ be distinct closed points on $E$. Let $V_1=V(\mathcal O_C(a))$ and $V_2=V(\mathcal O_E(b+c))$ be line bundles over $C$ and $E$, respectively (using the notation of \cite[Exercise II.5.18]{H}). 
	Identify the zero sections of $V_1$ and $V_2$ with $C$ and $E$ respectively, so that 
	$\mathcal O_C\otimes\mathcal O_{V_1}(C)\cong \mathcal O_C(-a)$ and $\mathcal O_E\otimes\mathcal O_{V_2}(E)\cong \mathcal O_E(-b-c).$ Let $x_1,y_1$ be local equations at $q_1$ in $V_1$, where $x_1=0$ is a local equation for the fiber of $V_1\rightarrow C$ through $q_1$ and $y_1=0$ is a local equation of $C$ at $q_1$. Let $x_2,y_2$ be local equations at $q_2$ in $V_2$, where $x_2=0$ is a local equation for the fiber of $V_2\rightarrow E$ through $q_2$ and $y_2=0$ is a local equation of $E$ at $q_2$. 
	
	Using plumbing, as is explained in \cite[p. 83]{La}, we  construct a complex manifold $V$ which contains the curves $C$ and $E$ intersecting in a point $q$. Identify $x_1$ with $y_2$ and $y_1$ with $x_2$ in some small analytic neighborhoods of $q_1$ and $q_2$ to obtain a complex manifold $V$ which is a neighborhood of $C+E$ with $q_1$ and $q_2$ identified to a point $q$ on $V$. By construction,
	$\mathcal O_C\otimes\mathcal O_V(C)\cong \mathcal O_C(-a)$ and $\mathcal O_E\otimes\mathcal O_V(E)\cong \mathcal O_E(-b-c)$ and $C\cdot E=q$. The intersection matrix
	\begin{equation}\label{eqN11}
	\left(\begin{array}{cc}
		(C^2)&(C\cdot E)\\
		(E\cdot C)&(E^2)
	\end{array}\right)
	=
	\left(\begin{array}{cc}
		-1&1\\1&-2
	\end{array}\right)
	\end{equation}
	of $C$ and $E$ is negative definite, so by Grauert's contraction criterion \cite{Ga}, there is a bimeromorphic holomorphic map $g:V\rightarrow W$, where $W$ is a normal complex analytic surface and $g$ is the contraction of $C$ and $E$ to a point $p$ on $W$. 
	

Let $R=\mathcal O_{W,p}$, an analytic local ring. Since $R$ has an isolated singularity, by \cite[Theorem 1]{Hir} or \cite[Theorem A]{CS} there exists a two dimensional local domain $(A,\m_A)$ which is essentially of finite type over $\CC$ such that $\hat A\cong \hat R$.  Thus $R\cong A^{\rm an}$ by \cite[Corollary 1.6]{Art}.  Moreover, this isomorphism can be chosen so that it is induced by an isomorphism of a ring of convergent power series. 
That is, there exists a local ring $B=\CC[x_1,\ldots,x_n]_{\fm}$ with maximal ideal $\fm_B$, 
where $\CC[x_1,\ldots,x_n]$ is the polynomial ring with maximal homogeneous ideal $\fm=(x_1,\ldots,x_n)$, and ideals $I\subset B$ and $J\subset \hat B$ with $A\cong B/I$ and $\hat R\cong \hat B/J$, and such that there exists a $\CC$-isomorphism $\sigma$ of $\hat B$ such that  $\sigma(\hat I)=J$, so that $\sigma$ induces an isomorphism of $\hat A$ and $\hat R$.

The ring $A$ is normal since $\hat A$ is normal. So, 
$A$ is the local ring of a closed point $p'$ on a normal complex projective surface $V'$. Let $Z'\rightarrow V'$ be a projective birational morphism which is a resolution of singularities. Let $\alpha:(Z')^h\rightarrow (V')^h$ be the induced analytic morphism of complex analytic spaces. There exist analytic neighborhoods $U_p$ of $p$ in $W$ and $U_{p'}$ of $p'$ in $V'$ with an analytic isomorphism $\lambda:U_{p'}\rightarrow  U_{p}$, such that $\alpha^{-1}(U_{p'}\setminus \{p'\})\rightarrow U_{p'}\setminus\{p'\}$ is an isomorphism. Now $(\lambda\alpha)^{-1}(U_p)\rightarrow U_p$ and $g^{-1}(U_p)\rightarrow U_p$ are resolutions of singularities of $U_p$, which are isomorphisms over $U_p\setminus\{p\}$.

A $(-1)$-curve on a nonsingular complex analytic surface $S$, or on a nonsingular finite type surface $S$ over $\CC$, is a curve $C$ such that $C\cong \PP^1$ and $(C^2)=-1$ (equivalently $C$ is a proper curve over $\CC$ such that $(C^2)<0$ and $(K_S\cdot C)<0$,  where $K_S$ is a canonical divisor on $S$).

 Now $g^{-1}(U_p)\rightarrow U_p$ is the minimal resolution of singularities of $U_p$ since $g^{-1}(U_p)$ contains no $(-1)$-curves (the minimal resolution of singularities of a complex analytic surface is defined on \cite[page 85]{BPV}).
	Thus we have an analytic morphism 
	$(\lambda\alpha)^{-1}(U_p)\rightarrow g^{-1}(U_p)$ which is a contraction of a sequence of $(-1)$-curves that factors 
	$(\lambda\alpha)^{-1}(U_p)\rightarrow U_p$ by \cite[Theorem III.6.2 and Theorem III.6.3]{BPV}. 
	
 A $(-1)$-curve $F$ in $(\lambda\alpha)^{-1}(U_p)$ is a $(-1)$-curve on $Z'$, so we may blow down $F$ on $Z'$ to get a nonsingular surface $Z^*$ which is a nonsingular projective variety by Castelnuovo's criterion, \cite[Theorem V.5.7]{H}. 
 The birational map $Z^*\dashrightarrow V'$ is a morphism by Zariski's main theorem, \cite[Theorem V.5.2]{H}.  Let $\beta:(Z^*)^h\rightarrow (V')^h$ be the induced analytic morphism. The morphism
 $(\lambda\alpha)^{-1}(U_p)\rightarrow (\lambda\beta)^{-1}(U_p)$ is the contraction of $F$ and there is an induced morphism $(\lambda\beta)^{-1}(U_p)\rightarrow g^{-1}(U_p)$ since $ g^{-1}(U_p)\rightarrow U_p$ is the minimal resolution of singularities of $U_p$.
 Iterating, after a finite number of blowdowns of $(-1)$-curves, we obtain a nonsingular projective surface $X'$, with a birational morphism $X'\rightarrow V'$ such that $\pi:X:= X'\times_{V'}\Spec(A)\rightarrow \Spec(A)$ is such that $\pi^{-1}(\m_A)=C+E$ ,   and $C$ and $E$ have the intersection matrix \autoref{eqN11}. 

	Let $\Delta=C+E$, an exceptional divisor on $X$, and define the divisorial filtration of $\m_A$-primary ideals $\II=\{I_n\}_{n\gs 0}$ of $A$, where $I_n=\Gamma(X,\mathcal O_X(-n\Delta))$. We have
	$$\mathcal O_C\otimes\mathcal O_X(-\Delta)\cong \mathcal O_C(a-q) \qquad \text{and} \qquad\mathcal O_E\otimes\mathcal O_X(-\Delta)\cong \mathcal O_E(b+c-q)\cong\mathcal O_{\PP^1}(1),$$ and a short exact sequence for each $n>0$,
	$$
	0\rightarrow \mathcal O_X(-n\Delta-C)\rightarrow \mathcal O_X(-n\Delta)\rightarrow \mathcal O_X(-n\Delta)\otimes\mathcal O_C\rightarrow 0.
	$$
	Note that $\Gamma(C,\mathcal O_C\otimes\mathcal O_X(-n\Delta))=0$
    since 
    $$
    \mathcal O_C\otimes\mathcal O_X(-n\Delta)\cong\mathcal O_C(n(a-q))=\mathcal O_C(n(a-q_1))
    $$
    has infinite order in the Jacobian of $C$. 
    Thus  we have an identification $\Gamma(X,\mathcal O_X(-n\Delta))=\Gamma(X,\mathcal O_X(-n\Delta-C))$ for each $n>0$, and so  $C\subset \sBL(-\Delta)$.
    The stable base locus is defined in \autoref{eqN21}.
	
	For integers $m\gs 0$ and $n\gs 2$, we have that
	\begin{equation}\label{eq2E4}
		H^1(E,\mathcal O_E\otimes\mathcal O_X(-(m+1)C-mE-n\Delta-2C))=0
	\end{equation}
	since $$\deg(\mathcal O_E\otimes\mathcal O_X(-(m+1)C-mE-n\Delta-2C))= m-3+n,$$
     and for $m\gs 0$ and $n\gs 2$,
	\begin{equation}\label{eq2E5}
		H^1(C,\mathcal O_C\otimes\mathcal O_X(-mC-mE-n\Delta-2C))=0
	\end{equation}
	since $\deg(\mathcal O_C\otimes\mathcal O_X(-mC-mE-n\Delta-2C))=2>0$.

    For every integer $m\gs 0$ we have short exact sequences 
    \begin{align}\label{eq2E1}
    &0\rightarrow \O_E\otimes\mathcal O_X(-(m+1)C-mE)\rightarrow \mathcal O_{(m+1)C+(m+1)E}\rightarrow \mathcal O_{(m+1)C+mE}\rightarrow 0\\
   \nonumber&0\rightarrow \mathcal O_{C}\otimes\mathcal O_X(-mC-mE)\rightarrow \mathcal O_{(m+1)C+mE}\rightarrow \mathcal O_{mC+mE}\rightarrow 0. 
    \end{align}
	Tensoring \autoref{eq2E1} with $\mathcal O_X(-n\Delta-2C)$, taking cohomology, and using 
	\autoref{eq2E4} and \autoref{eq2E5}, we obtain 
	$$
	H^1(X,\mathcal O_{mC+mE}\otimes\mathcal O_X(-n\Delta-2C))=0
	$$
	for all $m\gs 0$ and $n\gs 2$. Thus for all $n\gs 2$,
	\begin{equation}\label{eq2E2}
		H^1(X,\mathcal O_Y(-n\Delta-2C))=\lim_{\leftarrow}H^1(X,\mathcal O_{mC+mE}\otimes\mathcal O_X(-n\Delta-2C))=0
	\end{equation}
	by  \cite[Theorem III.11.1]{H}, where the inverse limit is over $m\in \ZZ_{>0}$.

    Now $\mathcal O_C\otimes\mathcal O_X(-n\Delta-C)\cong \mathcal O_C(n(a-q)+a)$. Since $\mathcal O_C(a-q)$ has infinite order in the Jacobian of $C$, we have by the Riemann-Roch theorem that for $n\gs 1$ there exists a unique closed point $p_n\in C$ such that $\mathcal O_C\otimes\mathcal O_X(-n\Delta-C)\cong\mathcal O_C(p_n)$ and the $p_n$ are all distinct. 
	By \autoref{eq2E2}, for all $n$ we have a surjection
	\begin{equation}\label{eq2E3}
		\Gamma(X,\mathcal O_X(-n\Delta-C))\rightarrow \Gamma(C,\mathcal O_C\otimes\mathcal O_X(-n\Delta-C))\cong \Gamma(C,\mathcal O_C(p_n)),
	\end{equation}
	where $\dim_{\CC}\big(\Gamma(C,\mathcal O_C(p_n))\big)=1$  by the Riemann-Roch Theorem. 
	Write $I_n\mathcal O_X=\mathcal I_n\mathcal O_X(-n\Delta-C)$, where $\mathcal I_n$ is an ideal sheaf on $X$. 
    Since $\Gamma(X,\mathcal O_X(-n\Delta-C))=\Gamma(X,\mathcal O_X(-n\Delta))$, we have that $\Gamma(X,\mathcal O_X(-n\Delta-C))\mathcal O_X=\mathcal I_n\mathcal O_X(-n\Delta-C)$.

   Let $\sigma\in \Gamma(C,\mathcal O_C(p_n))$ be such that $\Div(\sigma)=p_n$ (with the notation of \autoref{SecCart}).   Then $\sigma$ induces an inclusion $\mathcal O_C\xrightarrow{\sigma}\mathcal O_C(p_n)$ so that there exists a generator $t$ of $\m_{p_n}$ where $\fm_{p_n}$ is the maximal ideal of $\mathcal O_{C,p_n}$,  such that 
$\mathcal O_{C,p_n} \xrightarrow{\sigma}\mathcal O_C(p_n)_{p_n}$ is the map
    $
    \mathcal O_{C,p_n} \xrightarrow{t}\frac{1}{t}\mathcal O_{C,p_n}.
    $ We have
    $p_a(C)=1$, since $C$ is an elliptic curve. Thus $h^0(\mathcal O_C(p_n))=1$ by the Riemann-Roch theorem, recalled in \autoref{SecRR}. Therefore the natural inclusion
    $
    \Gamma(C,\mathcal O_C)\stackrel{\sigma}{\rightarrow}\Gamma(C,\mathcal O_C(p_n))
    $ 
    is an isomorphism. Thus $\Gamma(C,\mathcal O_C(p_n))\mathcal O_{C,p_n}=\m_{p_n}\mathcal O_{C}(p_n)_{p_n}$
 and $\Gamma(C,\mathcal O_C(p_n))\mathcal O_{C,\alpha}=\mathcal O_C(p_n)_{\alpha}$ for $\alpha\in C\setminus\{p_n\}$. 
   By \autoref{eq2E3},
\begin{align*}
\mathcal I_n\mathcal O_{C}(p_n)_{p_n}=\mathcal I_n\mathcal O_X(-n\Delta-C)\mathcal O_{C,p_n}
&=\Gamma(X,\mathcal O_X(-n\Delta-C))\mathcal O_{C,p_n}\\
&= \Gamma(C,\mathcal O_C(p_n))\mathcal O_{C,p_n}=\m_{p_n}\mathcal O_{C}(p_n)_{p_n}
\end{align*}
and similarly, $\mathcal I_n\mathcal O_C(p_n)_{\alpha}=\mathcal O_C(p_n)_{\alpha}$ for $\alpha\in C\setminus\{p_n\}$. 
Thus 
$\mathcal I_n\subset 
\mathcal J_{p_n}$, where 
$\mathcal J_{p_n}$ is the ideal sheaf of $p_n$ on $X$, and  further, by \autoref{eq2E3}, $\mathcal I_{n,\alpha}=\mathcal O_{X,\alpha}$ for $\alpha\in C\setminus \{p_n\}$. In particular, $q\not\in \BL(\Gamma(X,\mathcal O_X(-n\Delta-C)))$ for $n\gg 0$. 
    Therefore  for $n\gg 0$, $E\not\subset \BL(\Gamma(X,\mathcal O_X(-n\Delta-C))$ and then $\Supp(\mathcal O_X/\mathcal I_n)$ is a finite set of points, which includes $p_n$. Thus for each $n\gg0$, there exists a Rees valuation $v_n$ of $I_n$ whose center on $X$ is $p_n$, and so the $v_n$ are distinct (inequivalent) valuations for $n\gg 0$.
	Thus the set of all Rees valuations of all the $I_n$ is an infinite set. 
    The non-finite generation of $A[\II]$  follows  from \autoref{prop:Rees_finite}. Alternatively, it follows from \cite[Lemma 8.3]{C4} since 
    the stable base locus of $-\Delta$ is $\sBL(-\Delta)=C$.
\end{proof}

\begin{center}
	{\it Acknowledgments}
\end{center}

We acknowledge support by the NSF Grant  DMS \#1928930, while the authors were in residence at the Simons Laufer Mathematical Science Institute (formerly MSRI) in Berkeley, California, during the Spring 2024 semester.

\bibliographystyle{plain}

\end{document}